\documentclass[12pt, twoside ]{article}
\usepackage{amsfonts}
\usepackage{amsmath,amsthm}
\usepackage{enumerate}
\usepackage{color}

\usepackage[OT1]{fontenc}
\usepackage{bm}
\usepackage{mathrsfs}
\allowdisplaybreaks[4]

\newtheorem{thm}{Theorem}[section]
\newtheorem{lem}[thm]{Lemma}
\newtheorem{rem}[thm]{Remark}
\newtheorem{prop}[thm]{Proposition}

\numberwithin{equation}{section}

\begin{document}

\baselineskip=1.4em

\title{\textbf{Approximate functional equation and mean value formula for the derivatives of \textit{L}-functions attached to cusp forms}}
\author{Yoshikatsu Yashiro\\
\small Graduate School of Mathematics, Nagoya University,\\[-0.4em] 
\small 464-8602 \ Chikusa-ku, Nagoya, Japan \\[-0.4em] 
\small E-mail: m09050b@math.nagoya-u.ac.jp
}
\date{}
\maketitle

\renewcommand{\thefootnote}{}
\footnote{2010 \emph{Mathematics Subject Classification}: Primary 11M99; Secondary 11N99.}
\footnote{\emph{Key words and phrases}: cusp forms, $L$-functions, derivative, approximate functional equation, mean value formula.}

\vspace{-3em}

\begin{abstract}
Let $f$ be a holomorphic cusp form of weight $k$ with respect to the full modular group $SL_2(\mathbb{Z})$. We suppose that $f$ is a normalized Hecke eigenform. Let $L_f(s)$ be the $L$-function attached to the form $f$. 
Good gave the approximate functional equation and mean square formula of $L_f(s)$. 
In this paper, we shall generalize these formulas for the derivatives of $L_f(s)$. 
\end{abstract}



\maketitle

\section{Introduction}
Let $S_k$ be the space of cusp forms of even weight $k\in\mathbb{Z}_{\geq12}$ with respect to the full modular group $SL_2(\mathbb{Z})$. Let $f\in S_k$ be a normalized Hecke eigenform, and $a_f(n)$ the $n$-th Fourier coefficient of $f$. Set $\lambda_f(n)=a_f(n)/n^{(k-1)/2}$. 
The $L$-function attached to $f$ is defined by
\begin{align}
L_f(s)=\sum_{n=1}^\infty\frac{\lambda_f(n)}{n^s}=\prod_{p{\rm:prime}}\left(1-\frac{\alpha_f(p)}{p^s}\right)^{-1}\left(1-\frac{\beta_f(p)}{p^s}\right)^{-1} \quad (\text{Re }s>1), \label{3LD}
\end{align}
where $\alpha_f(p)$ and $\beta_f(p)$ satisfy $\alpha_f(p)+\beta_f(p)=\lambda_f(p)$ and $\alpha_f(p)\beta_f(p)=1$. Then it is well-known that the function $L_f(s)$ is analytically continued to the whole $s$-plane by
\begin{align}
(2\pi)^{-s-\frac{k-1}{2}}\Gamma(s+\tfrac{k-1}{2})L_f(s)=\int_0^\infty f(iy)y^{s+\frac{k-1}{2}-1}dy, \label{3AC}
\end{align}
and has a functional equation 
\begin{align*}
L_f(s)=\chi_f(s)L_f(1-s)
\end{align*}
where $\chi_f(s)$ is given by
\begin{align}
\chi_f(s)=&(-1)^{\frac{k}{2}}(2\pi)^{2s-1}\frac{\Gamma(1-s+\frac{k-1}{2})}{\Gamma(s+\frac{k-1}{2})} \label{XFE} \\
=&(-1)^{\frac{k}{2}}(2\pi)^{2\sigma-1}|t|^{1-2\sigma}e^{i\left(\frac{\pi}{2}(1-k){\rm sgn}(t)-2t\log\frac{|t|}{2\pi e}\right)}(1+O(|t|^{-1})) \label{STR}
\end{align}
where ${\rm sgn}(t)$ is defined by ${\rm sgn}(t)=1$ for $t\in\mathbb{R}_{>0}$ and ${\rm sgn}(t)=-1$ for $t\in\mathbb{R}_{<0}$, and (\ref{STR}) is obtained by Stirling's formula (see \cite[(19)]{GD1}).

Good \cite{GD1} gave the approximate functional equation for $L_f(s)$:
\begin{align*}
L_f(\sigma+it)=\sum_{n\leq x}\frac{\lambda_f(n)}{n^{s}}+\chi_f(s)\sum_{n\leq y}\frac{\lambda_f(n)}{n^{1-s}}+O(|t|^{\frac{1}{2}-\sigma+\varepsilon})
\end{align*}
where $\varepsilon\in\mathbb{R}_{>0}$, $s=\sigma+it$ such that $\sigma\in[0,1]$ and $|t|\gg 1$, and $x, y\in\mathbb{R}_{>0}$ satisfying $(2\pi)^2xy=|t|^2$. The feature of his proof of this equation is to introduce characteristic function and use the residue theorem. Moreover, he gave the mean square formula for $L_f(s)$ using the above equation:
\begin{align}
\int_1^T|L_f(\sigma+it)|^2dt=\begin{cases} A_{f}T\log T+O(T), & \sigma= 1/2, \\ \displaystyle T\sum_{n=1}^\infty\frac{|\lambda_f(n)|^2}{n^{2\sigma}}+O(T^{2(1-\sigma)}), & 1/2<\sigma<1, \\ \displaystyle T\sum_{n=1}^\infty\frac{|\lambda_f(n)|^2}{n^{2\sigma}}+O(\log ^2T), & \sigma=1,  \end{cases} \label{MVF}
\end{align}  
where $A_{f}$ is a positive constant depending on $f$.

Let $\zeta(s)$ be the Riemann zeta function and $\zeta'(s)$ be its first derivative. Since Speiser \cite{SPE} proved that the Riemann Hypothesis (for short RH) is equivalent to the non-existence of zeros of $\zeta'(s)$ in $0<\text{Re }s<1/2$, zeros of $\zeta'(s)$ have been interested
by many researchers. Recently Aoki and Minamide \cite{A&M} studied the density of zeros of $\zeta^{(m)}(s)$ in the right hand side of critical line ${\rm Re\;}s=1/2$ by using Littlewood's method. 
However there is no result concerning zeros of derivatives of $L$-functions attached to cusp forms.
The $m$-th derivative of $L_f(s)$ is given by
\begin{align*}
L_f^{(m)}(s)=\sum_{n=1}^\infty\frac{\lambda_f(n)(-\log n)^m}{n^s} \quad (\text{Re }s>1).
\end{align*}
Differentiating both sides of (\ref{3AC}), we find 
\begin{align}
L^{(m)}_f(s)=\sum_{r=0}^m\binom{m}{r}(-1)^{r}\chi_f^{(m-r)}(s)L_f^{(r)}(1-s). \label{3DFE}
\end{align}
In this paper, we shall show the approximate functional equation and the mean value formula for $L_f^{(m)}(s)$ for the purpose of studying the zero-density for $L_f^{(m)}(s)$.

Following \cite{GD1}, we shall introduce characteristic functions. Let $\varphi$ be the real valued $C^\infty$ function on $[0,\infty)$ satisfying $\varphi(\rho)=1$ for $\rho\in[0,1/2]$ and $\varphi(\rho)=0$ for $\rho\in[2,\infty)$. Let $\mathcal{R}$ be the set of these characteristic functions $\varphi$. Write $\varphi_0(\rho)=1-\varphi(1/\rho)$. It is clear to show that if $\varphi\in\mathcal{R}$ then $\varphi_0\in\mathcal{R}$. Let $\varphi^{(j)}$ be the $j$-th derivative function of $\varphi\in\mathcal{R}$. Then $\varphi^{(j)}$ becomes absolutely integrable function on $[0,\infty)$. Let $\|\varphi^{(j)}\|_1$ be $L_1$-norm of $\varphi^{(j)}$, that is, $\|\varphi^{(j)}\|_1=\int_0^\infty|\varphi^{(j)}(\rho)|d\rho$.
For $r\in\{0,\dots,m\}$, $j\in\mathbb{Z}_{\geq0}$, $\rho\in\mathbb{R}_{>0}$ and $s=\sigma+it$ such that $|t|\gg1$, let $\gamma_{j}^{(r)}(s,\rho)$ be
\begin{align*}
\gamma_{j}^{(r)}(s,\rho)&=\frac{1}{2\pi i}\int_{\mathcal{F}}\frac{(\chi_f^{(r)}/\chi_f)(1-s-w)}{w(w+1)\cdots(w+j)}\frac{\Gamma(s+w+\frac{k-1}{2})}{\Gamma(s+\frac{k-1}{2})}(\rho e^{-i\frac{\pi}{2}{\rm sgn}(t)})^wdw
\end{align*}
where $\mathcal{F}$ is given by $\mathcal{F}=\{-1/2-\sigma+\sqrt{|t|}e^{i\pi\theta}\mid\theta\in(1/2,3/2)\}\cup\{3/2-\sigma+\sqrt{|t|}e^{i\pi\theta}\mid\theta\in(-1/2,1/2)\}\cup\{u\pm\sqrt{|t|}\mid u\in[-1/2-\sigma,3/2-\sigma]\}$.  

Then using (\ref{3DFE}) and the approximate formula for $\chi_f^{(r)}(s)$ as $|t|\to\infty$ where $r\in\{0,\dots,m\}$, we obtain the approximate functional equation for $L_f^{(m)}(s)$ with characteristic functions:
\begin{thm}\label{ATH1}
For any $m\in\mathbb{Z}_{\geq0}$, $l\in\mathbb{Z}_{\geq(k+1)/2}$, $\varphi\in\mathcal{R}$, $s=\sigma+it$ such that $\sigma\in[0,1]$ and $|t|\gg 1$,  and $y_1,y_2\in\mathbb{R}_{>0}$ satisfying $(2\pi)^2y_1y_2=|t|^2$, we have
\begin{align}
L_f^{(m)}(s)=&\sum_{n=1}^\infty\frac{\lambda_f(n)(-\log n)^m}{n^s}\varphi\left(\frac{n}{y_1}\right)+\notag\\
&+\sum_{r=0}^m(-1)^{r}\binom{m}{r}\chi_f^{(m-r)}(s)\sum_{n=1}^\infty\frac{\lambda_f(n)(-\log n)^{r}}{n^{1-s}}\varphi_0\left(\frac{n}{y_2}\right)+R_\varphi(s), \label{AA1}
\end{align}
where $R_\varphi(s)$ is given by
\begin{align*}
R_\varphi(s)=&\sum_{n=1}^\infty\frac{\lambda_f(n)(-\log n)^m}{n^s}\sum_{j=1}^l\varphi^{(j)}\left(\frac{n}{y_1}\right)\left(-\frac{n}{y_1}\right)^{j}\gamma_{j}^{(0)}\left(s,\frac{1}{|t|}\right)+\\
&+\chi_f(s)\sum_{r=0}^m(-1)^j\binom{m}{r}\sum_{n=1}^\infty\frac{\lambda_f(n)(-\log n)^r}{n^{1-s}}
\times\\
&\times
\sum_{j=1}^l\varphi_0^{(j)}\left(\frac{n}{y_2}\right)\left(-\frac{n}{y_2}\right)^{j}\gamma_{j}^{(m-r)}\left(1-s,\frac{1}{|t|}\right)+\\
&+O\left(y_1^{1-\sigma}(\log y_1)^m|t|^{-\frac{l}{2}}\|\varphi^{(l+1)}\|_1\right)+\\
&+O\left(y_2^{\sigma}\left({\displaystyle\sum\limits_{r=0}^m}(\log y_2)^r(\log|t|)^{m-r}\right)|t|^{1-2\sigma-\frac{l}{2}}\|\varphi_0^{(l+1)}\|_1\right).
\end{align*}
\end{thm}
Introducing new functions $\xi\not\in\mathcal{R}$ and $\psi_\alpha\in\mathcal{R}$ for making the main term of  without characteristic function and the error term depending on $\alpha\in\mathbb{R}_{\geq0}$ of the approximate functional equation, replacing $\varphi$ to $\varphi_\alpha$ in Theorem \ref{ATH1}, using Deligne's result (see \cite{DEL}): $|\lambda_f(n)|\leq d(n)$ and choosing $\alpha$ to minimize the error term, we obtain the approximate functional equation for $L_f^{(m)}(s)$:
\begin{thm}\label{ATH2}
For any $m\in\mathbb{Z}_{\geq0}$ and $s=\sigma+it$ such that $\sigma\in[0,1]$ and $|t|\gg 1$, we have
\begin{align}
L_f^{(m)}(s)=&\sum_{n\leq \frac{|t|}{2\pi}}\frac{\lambda_f(n)(-\log n)^m}{n^s}\notag+\\
&+\sum_{r=0}^m(-1)^{r}\binom{m}{r}\chi_f^{(m-r)}(s)\sum_{n\leq \frac{|t|}{2\pi}}\frac{\lambda_f(n)(-\log n)^r}{n^{1-s}}+O(|t|^{1/2-\sigma+\varepsilon}), 
\end{align}
where $\varepsilon$ is an arbitrary positive number.
\end{thm}

Using Rankin's result (see \cite[(4.2.3), p.364]{RAN}):
\begin{align}
\sum_{n\leq x}|\lambda_f(n)|^2=C_fx+O(x^{\frac{3}{5}}) \label{RAN}
\end{align}
where $C_f$ is a positive constant depending on $f$, the approximate formula of $\chi_f^{(r)}(s)$ and the generalizations of Lemmas 6, 7 of \cite{GD1} to estimate a double sum containing $(\log n_1)^{r_1}(\log n_2)^{r_2}$ where $r_1+r_2=r$, we obtain the mean square for $L^{(m)}_f(s)$:
\begin{thm}\label{ATH3}
For any $m\in\mathbb{Z}_{\geq0}$ and large $T\in\mathbb{R}_{>0}$,  we have
\begin{align}
&\int_0^T|L_f^{(m)}(\sigma+it)|^2dt\notag\\
&=\begin{cases} A_{f,m}T(\log T)^{2m+1} +O(T(\log T)^{2m}), & \sigma=1/2, \\\displaystyle T\sum_{n=1}^\infty \frac{|\lambda_f(n)|^2(\log n)^{2m}}{n^{2\sigma}}+O(T^{2(1-\sigma)}(\log T)^{2m}), & 1/2<\sigma<1, \\\displaystyle T\sum_{n=1}^\infty \frac{|\lambda_f(n)|^2(\log n)^{2m}}{n^{2\sigma}}+O((\log T)^{2m+2}), & \sigma=1, \end{cases} \label{AA3}
\end{align}
where $A_{f,m}$ is given by $$A_{f,m}=\Biggl(\frac{1}{2m+1}+\sum_{r=0}^{2m}\frac{(-2)^{2m-r}}{r+1}\sum_{r_1+r_2=r}\binom{m}{r_1}\binom{m}{r_2}\Biggr)C_f.$$
\end{thm}

Theorems \ref{ATH1}--\ref{ATH3} is applied to the study of zero-density estimate for $L_f^{(m)}(s)$ in \cite{YAS}. 
In order to prove Theorems \ref{ATH1}--\ref{ATH3}, we  shall show preliminary lemmas in Section \ref{AMPL}. Using these lemmas we shall give proof of Theorems \ref{ATH1}--\ref{ATH3} in Sections \ref{AMP1}--\ref{AMP3} respectively.

\section{Preliminary Lemmas}\label{AMPL}
To prove Theorem \ref{ATH1}, we introduce a new function.
For $\varphi\in\mathcal{R}$, let $K_\varphi(w)$ be the function
\begin{align*}
K_\varphi(w)=w\int_0^\infty\varphi(\rho)\rho^{w-1}d\rho \quad (\text{Re }w>0).
\end{align*}
Then the following fact is known:
\begin{lem}[{\cite[p.335, Lemma 3]{GD1}}]\label{EXE}
The function $K_\varphi(w)$ is analytically continued for to the whole $w$-plane, and has the functional equation
\begin{align}
K_{\varphi}(w)=&K_{\varphi_0}(-w). \label{KFE}
\end{align}
Furthermore we have the integral representation
\begin{align}
\frac{K_\varphi(w)}{w}=&\frac{(-1)^{l+1}}{w(w+1)\cdots(w+l)}\int_0^\infty\varphi^{(l+1)}(\rho)\rho^{w+l}d\rho \label{KEP}
\end{align}
for $l\in\mathbb{Z}_{\geq0}$. Especially $K_\varphi(0)=1$.
\end{lem}

Next the following fact is useful for estimating the integrals \eqref{HE0}, $I_1'$ and $I_2'$ in Section \ref{ATH1}: 
\begin{lem}[{\cite[p.334, Lemma 2]{GD1}}]
Put $s=\sigma+it$ and $w=u+iv$. For $c_1,c_2\in\mathbb{R}$ let $D_1$ be the strip such that $\sigma\in[c_1,c_2]$ and $t\in\mathbb{R}$ in $s$-plane, and $D_2$ a half-strip such that $\sigma\in(-\infty,-1/2-(k-1)/2)$ and $t\in(-1,1)$. 
For fixed $c_3, c_4\in\mathbb{R}_{>0}$, there exist $c_5\in\mathbb{R}_{>0}$ and $c_6\in\mathbb{R}_{>0}$ such that
\begin{align}
&\left|\frac{\Gamma(s+w+\frac{k-1}{2})}{\Gamma(s+\frac{k-1}{2})}(e^{-i\frac{\pi}{2}{\rm sgn}(t)})^w\right|\notag\\
&\leq \begin{cases} c_5\dfrac{(1+|t+v|)^{\sigma+u-\frac{1}{2}+\frac{k-1}{2}}}{|t|^{\sigma-\frac{1}{2}+\frac{k-1}{2}}}, & s\in D_1, \ s+w\in D_1\setminus D_2, \ |t|\geq c_3, \\ c_6|t|^u, & s\in D_1, \ |w|\leq c_4|t|^{1/2}. \end{cases} \label{GEE}
\end{align}
\end{lem}
The following fact is required to obtain the approximate formula for $(\chi_f^{(r)}/\chi_f)(s)$:
\begin{lem}\label{LIB}
Let $F$ and $G$ be holomorphic function in the region $D$ such that $F(s)\ne0$ and $\log F(s)=G(s)$ for $s\in D$. Then for any fixed $r\in\mathbb{Z}_{\geq1}$, there exist $l_1,\cdots,l_r\in\mathbb{Z}_{\geq0}$ and $C_{(l_1,\cdots,l_r)}\in\mathbb{Z}_{\geq0}$ such that
\begin{align}
\frac{F^{(r)}}{F}(s)
=\sum_{1l_1+\cdots+rl_r=r}C_{(l_1,\cdots,l_r)}(G^{(1)}(s))^{l_1}\cdots(G^{(r)}(s))^{l_r} \label{LII}
\end{align}
for $s\in D$. Especially $C_{(r,0,\cdots,0)}=1$. 
\end{lem}
\begin{proof}
The case $r=1$ is true because of $(F'/F)(s)=G'(s)$ for $s\in D$. If we assume (\ref{LII}) and $C_{(r,0\cdots,0)}=1$, then we have 
\begin{align*}
F^{(r+1)}(s)=& \sum_{1l_1+\cdots+rl_r=r} C_{(l_1,\cdots,l_r)}\Bigl((F'G^{(1)l_1}\cdots G^{(r)l_r})(s)+\\ 
&+l_1(FG^{(1)l_1-1}G^{(2)l_2+1}\cdots G^{(r)l_r})(s)+\cdots+\\
&+l_{r-1}(FG^{(1)l_1}\cdots G^{(r-1)l_{r-1}-1}G^{(r)l_r+1})(s)+\\
&+l_r(FG^{(1)l_1}\cdots G^{(r)l_r-1}G^{(r+1)})(s)\Bigr)\\
=& F(s)\sum_{1l_1'+\cdots+(r+1)l_{r+1}=r+1} C'_{(l_1',\cdots,l_{r+1}')}(G^{(1)l_1'}\cdots G^{(r+1)l_{r+1}'})(s)
\end{align*} 
and $C'_{(r+1,0,\cdots,0)}=1\cdot C(r,0,\cdots,0)=1$. Hence (\ref{LII}) is true for all $r\in\mathbb{Z}_{\geq1}$.
\end{proof}
Using Lemma \ref{LIB}, we can get the approximate formula for $(\chi_f^{(r)}/\chi_f)(s)$  as follows:
\begin{lem}\label{XAE}
For any $r\in\mathbb{Z}_{\geq1}$, the function $(\chi_f^{(r)}/\chi_f)(s)$ is holomorphic in $D=\mathbb{C} \setminus \{z\in \mathbb{C} \mid |\sigma|\geq k/2-1, |t|\leq 1/2 \}$. For any $s\in D$ we have
\begin{align*}
\frac{\chi_f^{(r)}}{\chi_f}(s)= 
\begin{cases} \displaystyle \left(-2\log\frac{|t|}{2\pi}\right)^r+O\left(\frac{(\log|t|)^{r-1}}{|t|}\right), & |t|\gg1, \\ O(1), & \text{$|t|\ll1$.} \end{cases} 
\end{align*}
\end{lem}
\begin{proof}
Apply Lemma \ref{LIB} with $F(s)=\chi_f(s)$ and $G(s)=k\log i+(2s-1)\log{2\pi}+\log\Gamma(1-s+\frac{k-1}{2})-\log\Gamma(s+\frac{k-1}{2})$. Then we have 
\begin{align}
&\hspace{-1em}G^{(1)}(s)\notag\\
=&2\log{2\pi}-\frac{\Gamma'}{\Gamma}(1-s+\tfrac{k-1}{2})-\frac{\Gamma'}{\Gamma}(s+\tfrac{k-1}{2})\notag\\
=&-\log(s+\tfrac{k-1}{2})-\log(1-s+\tfrac{k-1}{2})+\frac{1}{2(s+\tfrac{k-1}{2})}+\frac{1}{2(1-s+\tfrac{k-1}{2})}+ \label{G1F}\\
&+2\log 2\pi+\int_0^\infty\frac{1/2-\{u\}}{(u+s+\frac{k-1}{2})^2}du+\int_0^\infty\frac{1/2-\{u\}}{(u+1-s+\frac{k-1}{2})^2}du\notag\\
=&\begin{cases} \displaystyle -2\log|t|+2\log{2\pi}+O\left(|t|^{-1}\right), & |t|\gg1, \\ O(1), & |t|\ll1 \end{cases} \notag
\end{align}
for $s\in D$ where we used the following formula obtained by Stirling's formula (see \cite[p.342, Theorem A.3.5]{KAR}):
\begin{align*}
\frac{\Gamma'}{\Gamma}(s)=\log s-\frac{1}{2s}-\int_0^\infty\frac{1/2-\{u\}}{(u+s)^2}du
\end{align*}
and the following the approximate formula (see \cite[p.335]{GD1}):
\begin{align*}
\log s=\log|t|+i\frac{\pi}{2}{\rm sgn\;}t+O\left(\frac{1}{|t|}\right), \quad \frac{1}{s}=-\frac{i}{t}+O\left(\frac{1}{|t|^2}\right).
\end{align*}
By differentiating both sides of (\ref{G1F}), for any $j\in\mathbb{Z}_{\geq2}$ and $s\in D$, $G^{(j)}(s)$ is approximated as $G^{(j)}(s)\ll 1/|t|^{j-1}$ when $|t|\gg1$ or $G^{(j)}(s)\ll 1$ when $|t|\gg1$. Since $C_{(r,0,\cdots,0)}=1$, it follows that the main term of $(\chi_f^{(r)}/\chi_f)(s)$ becomes $(G^{(1)}(s))^r$.
\end{proof}

In order to prove Theorem \ref{ATH3}, that is, to obtain the approximate formula of the mean square for $L_f^{(m)}(s)$ as sharp as possible, we divide the characteristic function $\varphi$ as a sum of $\varphi_1$ and $\varphi_2$. For $\varphi\in\mathcal{R}$, $\delta, \delta_1\in(0,1/2)$  such that $\delta<\delta_1<\delta_2$ where $\delta_2=2$, $\varphi_1$ and $\varphi_2$ are defined by
\begin{align}
\varphi_1(\rho)=\begin{cases} 1,  & \rho\in[0,\delta], \\ 0, & \rho\in[\delta_1,\infty), \end{cases} \quad
\varphi_2(\rho)=\begin{cases} 0,  & \rho\in[0,\delta], \\ 1, & \rho\in[\delta_1,1/2],\\ \varphi(\rho) & \rho\in[1/2,\delta_2],  \\ 0, & \rho\in[\delta_2,\infty), \end{cases} \label{CFD}
\end{align}
satisfying $(\varphi_1+\varphi_2)(\rho)=1$ for $\rho\in[\delta,\delta_1]$.
Similarly for $\varphi_0\in\mathcal{R}$,  $\varphi_{01}$ and $\varphi_{02}$ are defined by the above, where $\delta_{01}=\delta_1$ and $\delta_{02}=\delta_2=2$. 
We shall generalize Lemma 7 of p.351 in \cite{GD1}:
\begin{lem}[]\label{GIE}
Fix $\alpha\in\mathbb{Z}_{\geq0}$ and $\beta\in\mathbb{R}_{\geq0}$.
\begin{enumerate}[(a)]
\item \label{XXE} For $X\in\{1,01\}$, we have
\begin{align*}
&\int_1^T\overline{\varphi_X\left(\frac{2\pi n}{t}\right)}\varphi_X\left(\frac{2\pi n}{t}\right)\frac{\left(\log\frac{t}{2\pi}\right)^{\alpha}}{t^{\beta}}dt \\
&=
\begin{cases} 
T^{1-\beta}(\log{T})^\alpha/(1-\beta)+O\left((n^{1-\beta}\log n+T^{1-\beta})(\log T)^{\alpha-1}\right),\hspace{-14em} & \\ 
 & n\in[1,\delta T/2\pi),\; \beta\in[0,1),\; \alpha\in\mathbb{Z}_{\geq1},\\
T^{1-\beta}/(1-\beta)+O(n^{1-\beta}), & n\in[1,\delta T/2\pi),\; \beta\in[0,1),\; \alpha=0, \\
O\left(|\log(T/n)|(\log T)^{\alpha}\right), & n\in[1,\delta T/2\pi),\; \beta=1, \\
O((\log n)^{\alpha}/n^{\beta-1}), & n\in[1,\delta T/2\pi),\; \beta\in(1,\infty),\\
O(n^{1-\beta}(\log n)^\alpha), & n\in[\delta T/2\pi,\delta_1 T/2\pi), \\
0, & n\in[\delta_1T/2\pi,\infty), 
\end{cases}
\end{align*}
\item \label{XYE} For $X\in\{1,2\}$ and $Y\in\{2,02\}$, we have 
\begin{align*}
&\int_1^T \overline{\varphi_X\left(\frac{2\pi n}{t}\right)}\varphi_Y\left(\frac{2\pi n}{t}\right)\frac{\left(\log\frac{t}{2\pi}\right)^\alpha}{t^\beta} dt\\
&= 
\begin{cases}
O(n^{1-\beta}(\log n)^\alpha), & n\in[1,\delta_XT/2\pi), \\ 
 0, & n\in[\delta_XT/2\pi,\infty), \hspace{12em}
\end{cases}
\end{align*}
\item \label{XYN} For $X,Y\in\{1,2,01,02\}$ and $n_1\ne n_2$, we have 
\begin{align*}
&\int_1^T\overline{\varphi_X\left(\frac{2\pi n_1}{t}\right)}\varphi_Y\left(\frac{2\pi n_2}{t}\right)\left(\frac{n_1}{n_2}\right)^{it}\frac{\left(\log\frac{t}{2\pi}\right)^\alpha}{t^{\beta}} dt\\
&=\begin{cases}
0, & n_1\in[\delta_XT/2\pi,\infty), \\
0, & n_2\in[\delta_YT/2\pi,\infty), \\
\displaystyle\frac{\left(\log\frac{T}{2\pi}\right)^\alpha}{iT^\beta}\overline{\varphi_X\left(\frac{2\pi n_1}{T}\right)}\varphi_Y\left(\frac{2\pi n_2}{T}\right)\frac{(n_1/n_2)^{iT}}{\log(n_1/n_2)}+ \hspace{-1em} & \\[0.75em]
+O\left(\dfrac{(\log(\max\{n_1,n_2\}))^\alpha}{({\max\{n_1,n_2}\})^{1+\beta}((\log(n_1/n_2))^{2}}\right), &  n_1, n_2{\rm: otherwise,} \end{cases}
\end{align*}
\item \label{XYM} If there exist $\alpha\in\mathbb{Z}_{\geq0}$ and $\beta\in\mathbb{R}_{\geq0}$ such that $M(t)=O((\log t)^\alpha/t^{\beta})$, then for $X, Y\in\{1,2,01,02\}$ we have
\begin{align*}
&\int_1^T\overline{\varphi_{X}\left(\frac{2\pi n_1}{t}\right)}\varphi_{Y}\left(\frac{2\pi n_2}{t}\right)\left(\frac{n_1}{n_2}\right)^{it}M(t) dt\\
&=\begin{cases} 0, & n_1\in[\delta_XT/2\pi,\infty), \\ 0, &  n_2\in[\delta_YT/2\pi,\infty), \\ O(T^{1-\beta}(\log T)^\alpha), &  n_1, n_2{\rm: otherwise,}\\ & \beta\in[0,1), \\ 
O\left(|\log(T/\max\{n_1,n_2\})|(\log T)^{\alpha}\right), & n_1, n_2{\rm: otherwise,} \\ & \beta=1, \\ O\left((\log(\max\{n_1, n_2\}))^\alpha/(\max\{n_1, n_2\})^{\beta-1}\right), & n_1, n_2{\rm: otherwise,}\\ & \beta\in\mathbb{R}_{>1}. \end{cases}
\end{align*}
\item \label{XYS} For $X\in\{1,01\}, Y\in\{1,2,01,02\}$, we have
\begin{align*}
&\int_1^T\overline{\varphi_{X}\left(\frac{2\pi n_1}{t}\right)}\varphi_{Y}\left(\frac{2\pi n_2}{t}\right)(n_1n_2)^{it}\chi_f^{(\alpha)}(\sigma+it)dt\\
&=\begin{cases} 0, & n_1\in[\delta_XT/2\pi,\infty), \\ 0, &  n_2\in[\delta_YT/2\pi,\infty), \\ O\left(|\log(T/\max\{n_1,n_2\})|(\log T)^{\alpha}\right), & n_1, n_2{\rm:otherwise,}\\ & \sigma=1/2, \\ 
O\left((\log(\max\{n_1, n_2\}))^\alpha/(\max\{n_1, n_2\})^{2\sigma-1}\right), & n_1, n_2{\rm:otherwise,}\\ & \sigma\in(1/2,1]. \end{cases}
\end{align*}
\end{enumerate}
\end{lem}
\begin{proof}
First we consider the case $n_1\in[\delta_X T/2\pi,\infty)$ or $n_2\in[\delta_YT/2\pi,\infty)$. It is clear that $\varphi_X(2\pi n_1/t)=0$ or $\varphi_Y(2\pi n_2/t)=0$ for $t\in[1,T]$. 
Hence, (a)--(e) are true for the above $n_1, n_2$. 
Next we consider the case of $n_1\in[1,\delta_XT/2\pi)$ and $n_2\in[1,\delta_YT/2\pi)$. Then it is clear that $2\pi n_1/\delta_X$, $2\pi n_2/\delta_Y\in[1,T]$. For $t\in[1,2\pi\max(n_1/\delta_X,n_2/\delta_Y))$, we see that $\varphi_X(2\pi n_1/t)=0$ (if $n_1/\delta_X\geq n_2/\delta_Y$) or $\varphi_Y(2\pi n_2/t)=0$ (if $n_1/\delta_X\leq n_2/\delta_Y$). Hence,
\begin{align}
\int_1^T\cdots dt=\int_{2\pi\max(\frac{n_1}{\delta_X}, \frac{n_2}{\delta_Y})}^T\cdots dt. \label{IEA}
\end{align}
Later, we shall approximate the right-hand side of (\ref{IEA}).

First we consider the condition of (a), that is, $X,Y\in\{1,01\}$ and $n_1=n_2=:n$. When $n\in[\delta T/2\pi,\delta_1 T/2\pi)$, we see that $2\pi n/\delta\geq T$. Then the right-hand side of (\ref{IEA}) is estimated as
\begin{align}
\leq\int_{\frac{2\pi n}{\delta_1}}^{\frac{2\pi n}{\delta}}\left|\varphi_X\left(\frac{2\pi n}{t}\right)\right|^2\frac{(\log\frac{t}{2\pi})^\alpha}{t^\beta} dt\ll n^{1-\beta}(\log n)^{\alpha}. \label{IEB}
\end{align}
When $n\in[1,\delta T/2\pi)$, we find that $2\pi n/\delta\in[2\pi n/\delta_1,T]$ and $\varphi_X(2\pi n/t)=1$ for $t\in[2\pi n/\delta,T]$. Hence the right-hand side of (\ref{IEA}) is 
\begin{align}
=\int_{\frac{2\pi n}{\delta_1}}^{\frac{2\pi n}{\delta}}\left|\varphi_X\left(\frac{2\pi n}{t}\right)\right|^2\frac{(\log\frac{t}{2\pi})^\alpha}{t^\beta}dt+\int_{\frac{2\pi n}{\delta}}^T\frac{(\log\frac{t}{2\pi})^\alpha}{t^\beta}dt. \label{IEC}
\end{align}
Here the first term of the right-hand side on (\ref{IEC}) is estimated as
\begin{align}
\ll n^{1-\beta}(\log n)^\alpha, \label{IED}
\end{align}
the second term of the right-hand side on (\ref{IEC}) is
\begin{align}
=\begin{cases} O\left(|\log(T/n)|(\log T)^\alpha\right), & \beta=1, \\ T^{1-\beta}(\log T)^\alpha+O(T^{1-\beta}(\log T)^{\alpha-1})+O(n^{1-\beta}(\log n)^\alpha), & \beta\in[0,1), \\ O((\log n)^\alpha/n^{\beta-1})+O((\log T)^\alpha/T^{\beta-1}), & \beta\in\mathbb{R}_{>1}. \label{IEE}
 \end{cases}
\end{align}
where the following formula was used:
\begin{align}
&\int_M^N\frac{(\log\frac{t}{2\pi})^\alpha}{t^\beta}dt\notag\\
&=
\begin{cases} 
\displaystyle\frac{(\log\frac{N}{M})\left((\log\frac{N}{2\pi})^\alpha+(\log\frac{N}{2\pi})^{\alpha-1}(\log\frac{M}{2\pi})+\cdots+(\log\frac{M}{2\pi})^\alpha\right)}{\alpha+1}, & \beta=1,\\
\displaystyle\sum_{r=0}^\alpha\frac{(-1)^r}{(1-\beta)^{r+1}}\frac{\alpha!}{(\alpha-r)!}\left(\frac{(\log\frac{N}{2\pi})^{\alpha-r}}{N^{\beta-1}}-\frac{(\log\frac{M}{2\pi})^{\alpha-r}}{M^{\beta-1}}\right), & \beta\ne1.\\
\end{cases} \label{IEZ}
\end{align}
Therefore combining (\ref{IEA})--(\ref{IEE}), we obtain (a). 

Next we consider the condition (b), that is, $Y\in\{2,02\}$ and $n_1=n_2=n$. When $n\in[1,\delta_XT/2\pi)\cap[1,\delta_YT/2\pi)$, that is, $n\in[1,\delta_XT/2\pi)$, we see that $2\pi n/\delta\in[2\pi n/\delta_X,T]$ and $\varphi_Y(2\pi n/t)=0$ for $t\in[2\pi n/\delta,T]$. Then the right-hand side of (\ref{IEA}) is
\begin{align}
=\int_{\frac{2\pi n}{\delta_X}}^{\frac{2\pi n}{\delta}}\overline{\varphi_X\left(\frac{2\pi n}{t}\right)}\varphi_Y\left(\frac{2\pi n}{t}\right)\frac{(\log \frac{t}{2\pi})^\alpha}{t^\beta}dt\ll n^{1-\beta}(\log n)^\alpha. \label{IEF}
\end{align}
From (\ref{IEA}) and (\ref{IEF}), (b) is obtained.

We consider the condition of (c), that is, $n_1\ne n_2$, $n_1\in[1,\delta_XT/2\pi)$ and $n_2\in[1,\delta_YT/2\pi)$. By integral by parts, the right-hand side of (\ref{IEA}) is 
\begin{align}
=&\overline{\varphi_X\left(\frac{2\pi n_1}{T}\right)}\varphi_Y\left(\frac{2\pi n_1}{T}\right)\frac{(\log T)^\alpha}{T^\beta}\frac{(n_1/n_2)^{iT}}{i\log(n_1/n_2)}+\notag\\
&+\left(\overline{\varphi_X\left(\frac{2\pi n_1}{t}\right)}\varphi_Y\left(\frac{2\pi n_2}{t}\right)\frac{(\log t)^\alpha}{t^\beta}\right)'_{t=T}\frac{(n_1/n_2)^{iT}}{(\log(n_1/n_2))^2}-\notag\\
&-\frac{1}{(\log(n_1/n_2))^2}\int_{2\pi\max(\frac{n_1}{\delta_X},\frac{n_2}{\delta_Y})}^T\left(\overline{\varphi_X\left(\frac{2\pi n_1}{t}\right)}\varphi_Y\left(\frac{2\pi n_2}{t}\right)\frac{(\log t)^\alpha}{t^\beta}\right)''\times\notag\\&\times\left(\frac{n_1}{n_2}\right)^{it}dt. \label{IEG}
\end{align}
Since $(\varphi_X(2\pi n/t))'
=O(n/t^2)$ and $(\varphi_X(2\pi n/t))''=O(n/t^3)+O(n^2/t^4)$ for $X\in\{1,2,01,02\}$, it follows that
\begin{align*}
&(\cdots)_{t=T}'\ll(n_1+n_2)\frac{(\log T)^\alpha}{T^{\beta+2}}+\frac{(\log T)^{\alpha-1}}{T^{\beta+1}}+\frac{(\log T)^\alpha}{T^{\beta+1}}\ll\frac{(\log T)^\alpha}{T^{\beta+1}},\\
&(\cdots)''\ll(n_1+n_2)\frac{(\log t)^\alpha}{t^{\beta+3}}+(n_1^2+n_2^2)\frac{(\log t)^\alpha}{t^{\beta+4}}+n_1n_2\frac{(\log t)^\alpha}{t^{\beta+4}}\ll\frac{(\log t)^\alpha}{t^{\beta+2}}.
\end{align*}
Hence the second term of the right-hand side of (\ref{IEG}) is estimated as
\begin{align}
\ll\frac{(\log T)^\alpha}{T^{\beta+1}(\log(n_1/n_2))^2}\ll\frac{(\log\max(n_1,n_2))^\alpha}{(\max(n_1,n_2))^{\beta+1}(\log(n_1/n_2))^2}, \label{IEH}
\end{align}
and the third term of the right-hand side of (\ref{IEG}) is estimated as
\begin{align}
\ll\frac{1}{(\log(n_1/n_2))^2}\int_{2\pi\max(\frac{n_1}{\delta_X},\frac{n_2}{\delta_Y})}^T\frac{(\log t)^{\alpha}}{t^{\beta+2}}dt\ll\frac{(\log\max(n_1,n_2))^\alpha}{(\max(n_1,n_2))^{\beta+1}(\log(n_1/n_2))^2}. \label{IEI}
\end{align}
Combining (\ref{IEA}) and (\ref{IEG})--(\ref{IEI}), we obtain (c). 

Next we consider the condition of (d), that is, $n_1\in[1,\delta_XT/2\pi)$ and $n_2\in[1,\delta_YT/2\pi)$. Then (\ref{IEZ}) gives that the right-hand side of (\ref{IEA}) is estimated as
\begin{align*}
&\ll\int_{2\pi\max(\frac{n_1}{\delta_X},\frac{n_2}{\delta_Y})}^T\frac{(\log t)^\alpha}{t^\beta}dt\\
&\ll\begin{cases} T^{1-\beta}(\log T)^\alpha, & \beta\in[0,1), \\ |\log(T/\max(n_1,n_2))|(\log T)^\alpha, & \beta=1, \\ (\log\max(n_1,n_2)^\alpha)/(\max(n_1,n_2))^{\beta-1}, & \beta\in\mathbb{R}_{>1}. \end{cases}
\end{align*}
Thus (d) is obtained. 

Finally we consider the condition of (e), that is, $X\in\{1,01\}$, $n_1\in[1,\delta_XT/2\pi)$ and $n_2\in[1,\delta_YT/2\pi)$. 
Using (\ref{STR}) and Lemma \ref{XAE}, we have 
\begin{align}
(n_1n_2)^{it}\chi_f^{(\alpha)}(s)=&(-1)^{-\frac{k}{2}}(-2)^\alpha(2\pi)^{2\sigma-1}e^{i\frac{\pi}{2}(1-k){\rm sgn}(t)}
\times\notag\\
&\times
e^{-2t\log\frac{|t|}{2\pi e\sqrt{n_1n_2}}}|t|^{1-2\sigma}\left(\log\frac{|t|}{2\pi}\right)^\alpha+M_1(t) \label{IEJ}
\end{align}
where $M_1(t)=O((\log|t|)^\alpha/|t|^{2\sigma})$. Since we have $\delta_X\delta_Y<1$, it follows that $$2\pi\max(n_1/\delta_X,n_2/\delta_Y) \geq 2\pi\sqrt{(n_1n_2)/(\delta_X\delta_Y)}> 2\pi\sqrt{n_1n_2}.$$ Therefore we see that $|\log(2\pi\sqrt{n_1n_2}/t)|>-\log(\sqrt{\delta_X\delta_Y})>0$ and 
\begin{align}
e^{-i2t\log\frac{t}{2\pi e\sqrt{n_1n_2}}}=\left(\frac{e^{-i2t\log\frac{t}{2\pi e\sqrt{n_1n_2}}}}{2i\log\frac{2\pi\sqrt{n_1n_2}}{t}}\right)'-\frac{e^{-i2t\log\frac{t}{2\pi e\sqrt{n_1n_2}}}}{2it(\log\frac{2\pi\sqrt{n_1n_2}}{t})^2} \label{IEK}
\end{align}
for $t\in[2\pi\max(n_1/\delta_X,n_2/\delta_Y),T]$. 
By (\ref{IEJ}) and (\ref{IEK}), the right-hand side of (\ref{IEA}) is estimated as
\begin{align}
=&(-1)^{-\frac{k}{2}}(-2)^\alpha(2\pi)^{2\sigma-1}e^{i\frac{\pi}{2}(1-k)}\times\notag\\
&\times\int_{2\pi\max(\frac{n_1}{\delta_X},\frac{n_2}{\delta_Y})}^T\overline{\varphi_X\left(\frac{2\pi n}{t}\right)}\varphi_Y\left(\frac{2\pi n}{t}\right)\frac{(\log\frac{t}{2\pi})^\alpha}{t^{2\sigma-1}}\left(\frac{e^{-i2t\log\frac{t}{2\pi e\sqrt{n_1n_2}}}}{2i\log\frac{2\pi\sqrt{n_1n_2}}{t}}\right)'dt\notag\\
&+O\left(\int_{2\pi\max(\frac{n_1}{\delta_X},\frac{n_2}{\delta_Y})}^T\overline{\varphi_X\left(\frac{2\pi n}{t}\right)}\varphi_Y\left(\frac{2\pi n}{t}\right)M_2(t)dt\right), \label{IEL}
\end{align}
where $M_2(t)=O((\log t)^\alpha/t^{2\sigma})$.
From (d), the second term of the right-hand side of (\ref{IEL}) is estimated as
\begin{align}
\ll\begin{cases} |\log(T/\max(n_1,n_2))|(\log T)^\alpha, & \sigma=1/2, \\ (\log\max(n_1,n_2))^\alpha/(\max(n_1,n_2))^{2\sigma-1}, & \sigma\in(1/2,1]. \end{cases} \label{IEM}
\end{align}
Integration by parts and (\ref{IEZ}) give that the first term of the right-hand side of (\ref{IEL}) is
\begin{align}
=&\frac{(-2)^\alpha(2\pi)^{2\sigma-1}}{(-1)^{\frac{k}{2}}e^{i\frac{\pi}{2}(k-1)}}\left(\overline{\varphi_X\left(\frac{2\pi n}{T}\right)}\varphi_Y\left(\frac{2\pi n}{T}\right)\frac{(\log\frac{T}{2\pi})^\alpha}{T^{2\sigma-1}}\frac{e^{-i2t\log\frac{T}{2\pi e\sqrt{n_1n_2}}}}{2i\log\frac{2\pi\sqrt{n_1n_2}}{T}}+\right.\notag\\
&-\left.\int_{2\pi\max(\frac{n_1}{\delta_X},\frac{n_2}{\delta_Y})}^T\left(\overline{\varphi_X\left(\frac{2\pi n}{t}\right)}\varphi_Y\left(\frac{2\pi n}{t}\right)\frac{(\log\frac{t}{2\pi})^\alpha}{t^{2\sigma-1}}\right)'\frac{e^{-i2t\log\frac{t}{2\pi e\sqrt{n_1n_2}}}}{2i\log\frac{2\pi\sqrt{n_1n_2}}{t}}dt\right)\notag\\
\ll&\frac{(\log T)^\alpha}{T^{2\sigma-1}}+\int_{2\pi\max(\frac{n_1}{\delta_X},\frac{n_2}{\delta_Y})}^T\frac{(\log t)^{\alpha}}{t^{2\sigma}}dt\notag\\
\ll& \begin{cases} |\log(T/\max(n_1,n_2))|(\log T)^\alpha, & \sigma=1/2, \\ (\log\max(n_1,n_2))^\alpha/(\max(n_1,n_2))^{2\sigma-1}, & \sigma\in(1/2,1], \end{cases} \label{IEN}
\end{align}
where the following estimate was used:
\begin{align*}
(\cdots)'\ll (n_1+n_2)\frac{(\log t)^\alpha}{t^{2\sigma+1}}+\frac{(\log t)^\alpha}{t^{2\sigma}}+\frac{(\log t)^{\alpha-1}}{t^{2\sigma}}\ll\frac{(\log t)^\alpha}{t^{2\sigma}}.
\end{align*}
Combining (\ref{IEA}) and (\ref{IEL})--(\ref{IEN}), we get (e). 
\end{proof}

After using Lemma \ref{GIE}, we shall estimate the following sums:
\begin{lem}\label{G6A}\label{G6B}
For $x\in\mathbb{R}_{\geq2}$, $r_1, r_2\in\mathbb{Z}_{\geq0}$ and complex valued arithmetic functions $\alpha, \beta$ such that $\alpha(n)\ll|\lambda_f(n)|$, $\beta(n)\ll|\lambda_f(n)|$, we have
\begin{enumerate}[(a)]
\item $\displaystyle\sum_{n_1\leq n_2\leq x}\frac{|\lambda_f(n_1)\lambda_f(n_2)|(\log{n_1})^{r_1}(\log{n_2})^{r_2}}{(n_1n_2)^\sigma}\\ \ll\begin{cases} x^{2(1-\sigma)}(\log x)^{r_1+r_2}, & \sigma\in[1/2,1), \\ (\log x)^{r_1+r_2+2}, & \sigma=1, \end{cases}$
\item $\displaystyle\sum_{n_1\leq n_2\leq x}\frac{|\lambda_f(n_1)\lambda_f(n_2)|(\log{n_1})^{r_1}(\log{n_2})^{r_2}}{(n_1n_2)^\sigma}\left|\log\frac{x}{n_2}\right|\ll \frac{(\log x)^{r_1+r_2}}{x^{2(\sigma-1)}}\\$ for $\sigma\in[1/2,1)$, 
\item $\displaystyle \sum_{n_1<n_2\leq x}\frac{|\alpha(n_1)\beta(n_2)|(\log n_1)^{r_1}(\log n_2)^{r_2}}{(n_1n_2)^\sigma{n_2}(\log(n_1/n_2))^2}\\\ll
\begin{cases}
x^{2(1-\sigma)}(\log x)^{r_1+r_2}, & \sigma\in[1/2,1), \\ (\log x)^{r_1+r_2+2}, & \sigma=1,
\end{cases}$
\item $\displaystyle \sum_{\begin{subarray}{c}n_1,n_2\leq x,\\ n_1\ne n_2 \end{subarray}}
\frac{\overline{\alpha(n_1)}\beta(n_2)(\log n_1)^{r_1}(\log n_2)^{r_2}}{(n_1n_2)^\sigma\log(n_1/n_2)}\\\ll\begin{cases} x^{2(1-\sigma)}(\log x)^{r_1+r_2}, & \sigma\in[1/2,1), \\ (\log x)^{r_1+r_2+2}, & \sigma=1. \end{cases}$
\end{enumerate}
\end{lem}
\begin{proof}
Using the fact $(\log{n_1})^{r_1}(\log{n_2})^{r_2}\ll(\log x)^{r_1+r_2}$ for $n_1, n_2\leq x$ and the estimates of $R_\sigma(x)$ and $S_\sigma(x)$ in \cite[p.348, LEMMA 6]{GD1}, we obtain (a) and (b).
By the same discussion for $T_\sigma(x)$ and $U_\sigma(x)$ with $\alpha_{n_1}=\alpha(n_1)(\log n_1)^{r_1}$, $\beta_{n_2}=\beta(n_2)(\log n_2)^{r_2}$, $a_{n_1}=\lambda_f(n_1)(\log n_2)^{r_1}$, $b_{n_2}=\lambda_f(n_2)(\log n_2)^{r_2}$ in \cite[p.348, LEMMA 6]{GD1}, (c) and (d) are obtained.
\end{proof}

\section{Proof of Theorem \ref{ATH1}}\label{AMP1}
First we shall show the following formula: 
\begin{prop}\label{MAFE}
For $s=\sigma+it$ such that $\sigma\in[0,1]$ and $|t|\gg 1$, $\varphi\in\mathcal{R}$, $x\in\mathbb{R}_{>0}$, and fixed $l\in\mathbb{Z}_{\geq(l+1)/2}$, we have
\begin{align*}
L_f^{(m)}(s)=G_m(s,x;\varphi)+\chi_f(s)\sum_{r=0}^m(-1)^r\binom{m}{r}G_r\left(1-s,\frac{1}{x};\varphi_0\right),
\end{align*}
where $G_r(s,x;\varphi) \ (r\in\{0,\dots,m\})$ are given by
\begin{align*}
G_r(s,x;\varphi)
=&\frac{1}{2\pi i}\int_{(\frac{3}{2}-\sigma)}\frac{\chi_f^{(m-r)}}{\chi_f}(1-s-w)L_f^{(r)}(s+w)\frac{K_\varphi(w)}{w}\times\\
&\times\frac{\Gamma(s+w+\frac{k-1}{2})}{\Gamma(s+\frac{k-1}{2})}\left(\frac{x}{2\pi}e^{-i\frac{\pi}{2}{\rm sgn\;}t}\right)^wdw. 
\end{align*}
\end{prop}
\begin{proof}
First we shall show that the integral 
\begin{align}
\frac{1}{2\pi i}\int_{-\frac{1}{2}-\sigma\pm iv}^{\frac{3}{2}-\sigma\pm iv}L_f^{(m)}(s+w)\frac{K_\varphi(w)}{w}\frac{\Gamma(s+w+\frac{k-1}{2})}{\Gamma(s+\frac{k-1}{2})}\left(\frac{x}{2\pi}e^{-i\frac{\pi}{2}{\rm sgn\;}t}\right)^wdw \label{HE0}
\end{align}
vanishes as $|v|\to\infty$ for $l\in\mathbb{Z}_{\geq (k+1)/2}$.  Write $w=u+iv$ and choose $|v|\gg|t|+1$, then $|s+w|\gg|t+v|\gg1$. 
Using (\ref{STR}), (\ref{3DFE}) and Lemma \ref{XAE} we have
\begin{align*}
L_f^{(m)}(s)\ll\sum_{r=0}^m|t|^{1-2\sigma}(\log|t|)^{m-r}|L_f^{(r)}(1-s)|\ll|t|^{1-2\sigma}(\log|t|)^m \quad (|t|\to\infty) 
\end{align*}
for $\text{Re\;}s<0$.
Hence the Phragm\'{e}n-Lindel\"{o}f theorem gives 
\begin{align}
L_f^{(m)}(s+w)&\ll|t+v|^{\frac{3}{2}-(\sigma+u)}(\log |t+v|)^m\notag\\&\ll|v|^{\frac{3}{2}-(\sigma+u)}(\log|v|)^m \quad (|v|\to\infty) \label{HE1}
\end{align}
uniformly for $\sigma+u\in[-1/2,3/2]$. Using (\ref{KEP}) and (\ref{GEE}) we see that
\begin{align}
&\frac{K_\varphi(w)}{w}\times\frac{\Gamma(s+w+\frac{k-1}{2})}{\Gamma(s+\frac{k-1}{2})}\left(\frac{x}{2\pi}e^{-i\frac{\pi}{2}{\rm sgn\;}t}\right)^w\notag\\
&\ll\frac{\|\varphi^{(l+1)}\|_1}{|v|^{l+1}}\times\frac{(1+|t+v|)^{\sigma+u-\frac{1}{2}+\frac{k-1}{2}}}{|t|^{\sigma-\frac{1}{2}+\frac{k-1}{2}}}\ll|v|^{\sigma+u-\frac{3}{2}+\frac{k-1}{2}-l} \quad (|v|\to\infty) \label{HE2}
\end{align}
uniformly for $\sigma+u\in[-1/2,3/2]$. From (\ref{HE1}) and (\ref{HE2}), the integral (\ref{HE0}) is $\ll |v|^{\frac{k-1}{2}-l}(\log|v|)^m$, that is,  (\ref{HE0}) tends to $0$ as $|v|\to\infty$ when $l\in\mathbb{Z}_{\geq(k+1)/2}$.

Using the above fact, $K_\varphi(0)=1$ and applying Cauchy's residue theorem, we have
\begin{align}
L_f^{(m)}(s)=&\frac{1}{2\pi i}\left(\int_{(\frac{3}{2}-\sigma)}-\int_{(-\frac{1}{2}-\sigma)}\right)L_f^{(m)}(s+w)\frac{K_\varphi(w)}{w}\frac{\Gamma(s+w+\frac{k-1}{2})}{\Gamma(s+\frac{k-1}{2})}\times\notag\\
&\times\left(\frac{x}{2\pi}e^{-i\frac{\pi}{2}{\rm sgn\;}t}\right)^wdw. 
\label{AF1}
\end{align}
for $l\in\mathbb{Z}_{\geq(k+1)/2}$. Clearly, the first term of the right-hand side of (\ref{AF1}) is
\begin{align}
=G_m(s,x;\varphi).
\end{align} 
We consider the second term of the right-hand side of (\ref{AF1}). Now we can calculate   
\begin{align}
&\hspace{-1em}L_f^{(m)}(s+w)\frac{\Gamma(s+w+\frac{k-1}{2})}{\Gamma(s+\frac{k-1}{2})}\notag\\
=&\frac{\Gamma(s+w+\frac{k-1}{2})}{\Gamma(s+\frac{k-1}{2})}\chi_f(s+w)\sum_{r=0}^m(-1)^{r}\binom{m}{r}\frac{\chi_f^{(m-r)}}{\chi_f}(s+w)L_f^{(r)}(1-s-w)\notag\\
=&\chi_f(s)(2\pi)^{2w}\frac{\Gamma(1-s+w+\frac{k-1}{2})}{\Gamma(1-s+\frac{k-1}{2})}\sum_{r=0}^m(-1)^{r}\binom{m}{r}\frac{\chi_f^{(m-r)}}{\chi_f}(s+w)\times \label{LGG}\\
&\times L_f^{(r)}(1-s-w) \notag
\end{align}
where we used (\ref{XFE}) and (\ref{3DFE}) which give that  
\begin{align*}
\frac{\chi_f(s+w)}{\chi_f(s)}=(2\pi)^{2w}\frac{\Gamma(s+w)}{\Gamma(s+w+\frac{k-1}{2})}\frac{\Gamma(1-s-w+\frac{k-1}{2})}{\Gamma(1-s+\frac{k-1}{2})}. 
\end{align*}
Using (\ref{KFE}), (\ref{LGG}) and transforming $w\mapsto -w$, we see that the second term of the right-hand side of (\ref{AF1}) is
\begin{align}
=&-\frac{\chi_f(s)}{2\pi i}\int_{-(\frac{1}{2}+\sigma)}\frac{K_{\varphi}(-w)}{-w}\frac{\Gamma(1-s+w+\frac{k-1}{2})}{\Gamma(1-s+\frac{k-1}{2})}\left(2\pi xe^{-i\frac{\pi}{2}{\rm sgn}(t)}\right)^{-w}\times\notag\\
&\times \sum_{r=0}^m(-1)^r\binom{m}{r}\frac{\chi_f^{(m-r)}}{\chi_f}(s-w)L_f^{(r)}(1-s+w)(-dw)\notag\\
=&\chi_f(s)\sum_{r=0}^m(-1)^{r}\binom{m}{r}G_r\left(1-s,\frac{1}{x};\varphi_0\right).\label{AF2}
\end{align}
By (\ref{AF1})--(\ref{AF2}) Proposition \ref{MAFE} is showed. 
\end{proof}

Next, the approximate formula of $G_r(s,x;\varphi)$ is written as follows:
\begin{prop}\label{PAFE}
For $s=\sigma+it$ such that $\sigma\in[0,1]$ and $|t|\gg 1$, $\varphi\in\mathcal{R}$, $x, y\in\mathbb{R}_{>0}$ satisfying $x/(2\pi y)=1/|t|$, fixed $r\in\{0,\cdots,m\}$ and $l\in\mathbb{Z}_{\geq(k+1)/2}$, we have
\begin{align*}
G_r(s,x;\varphi)=&\sum_{n=1}^\infty\frac{\lambda_f(n)(-\log n)^r}{n^s}\sum_{j=0}^l\varphi^{(j)}\left(\frac{n}{y}\right)\left(-\frac{n}{y}\right)^j\gamma_{j}^{(m-r)}\left(s,\frac{1}{|t|}\right)+\\
&+O(y^{1-\sigma}(\log y)^r(\log|t|)^{m-r}|t|^{-\frac{l}{2}}\|\varphi^{(l+1)}\|_1).
\end{align*}
\end{prop}
\begin{proof}
First using (\ref{KEP}) and dividing the series $L_f^{(r)}(s+w)$ into two path at $\rho y$, we can write 
\begin{align}
G_r(s,x;\varphi)=I_{1}+I_{2}, \label{IC1}
\end{align}
where $I_1$ and $I_2$ are given by
\begin{align*}
I_1=&\frac{1}{2\pi i}\int_{(\frac{3}{2}-\sigma)}\frac{\Gamma(s+w+\frac{k-1}{2})}{\Gamma(s+\frac{k-1}{2})}\left(\frac{x}{2\pi}e^{-i\frac{\pi}{2}{\rm sgn\;}t}\right)^w
\frac{(-1)^l}{w(w+1)\cdots(w+l)}\times\\
&\times\frac{\chi_f^{(m-r)}}{\chi_f}(1-s-w)\left(\int_0^\infty\varphi^{(l+1)}(\rho)\rho^{w+l}\sum_{n\leq \rho y}\frac{\lambda_f(n)(-\log n)^r}{n^{s+w}}d\rho\right)dw,\\
I_2=&\frac{1}{2\pi i}\int_{(\frac{3}{2}-\sigma)}\frac{\Gamma(s+w+\frac{k-1}{2})}{\Gamma(s+\frac{k-1}{2})}\left(\frac{x}{2\pi}e^{-i\frac{\pi}{2}{\rm sgn\;}t}\right)^w
\frac{(-1)^l}{w(w+1)\cdots(w+l)}\times\\
&\times\frac{\chi_f^{(m-r)}}{\chi_f}(1-s-w)\left(\int_0^\infty\varphi^{(l+1)}(\rho)\rho^{w+l}\sum_{n>\rho y}\frac{\lambda_f(n)(-\log n)^r}{n^{s+w}}d\rho\right)dw.
\end{align*}
Let $L_{\pm1}, L_{\pm2}, C_1, C_2$ be paths of integration defined by \begin{align*}
&L_{\pm1}=\{-1/2-\sigma\pm iv \mid v\in(\sqrt{|t|},\infty)\}, \\
&L_{\pm2}=\{3/2-\sigma\pm iv \mid v\in(\sqrt{|t|},\infty)\},\\
&C_{1}=\{-1/2-\sigma+\sqrt{|t|}e^{-i\pi\theta} \mid \theta\in(1/2,3/2)\}, \\  
&C_{2}=\{3/2-\sigma+\sqrt{|t|}e^{i\pi\theta} \mid \theta\in(-1/2,1/2)\}.
\end{align*}
Then by the residue theorem, we have 
\begin{align}
I_{1}=I_{1}'+{\rm Res\;}\mathcal{F}, \quad  I_{2}=I_{2}', \label{IC3}
\end{align}
where $I_1', I_2', {\rm Res\:}\mathcal{F}$ are given by
\begin{align*}
I_1'=&\frac{1}{2\pi i}\int_{L_{-1}+C_{1}+L_{+1}}\frac{\Gamma(s+w+\frac{k-1}{2})}{\Gamma(s+\frac{k-1}{2})}\left(\frac{x}{2\pi}e^{-i\frac{\pi}{2}{\rm sgn\;}t}\right)^w\times\\
&\times\frac{(-1)^l}{w(w+1)\cdots(w+l)}\frac{\chi_f^{(m-r)}}{\chi_f}(1-s-w)\times\\
&\times\left(\int_0^\infty\varphi^{(l+1)}(\rho)\rho^{w+l}\sum_{n\leq \rho y}\frac{\lambda_f(n)(-\log n)^r}{n^{s+w}}d\rho\right)dw,\\
I_2'=&\frac{1}{2\pi i}\int_{L_{-2}+C_{2}+L_{+2}}\frac{\Gamma(s+w+\frac{k-1}{2})}{\Gamma(s+\frac{k-1}{2})}\left(\frac{x}{2\pi}e^{-i\frac{\pi}{2}{\rm sgn\;}t}\right)^w\times\\
&\times\frac{(-1)^l}{w(w+1)\cdots(w+l)}\frac{\chi_f^{(m-r)}}{\chi_f}(1-s-w)\times\\
&\times\left(\int_0^\infty\varphi^{(l+1)}(\rho)\rho^{w+l}\sum_{n>\rho y}\frac{\lambda_f(n)(-\log n)^r}{n^{s+w}}d\rho\right)dw,\\
{\rm Res\:}\mathcal{F}=&\sum_{w=0,-1,\dots,-l}\frac{\Gamma(s+w+\frac{k-1}{2})}{\Gamma(s+\frac{k-1}{2})}\left(\frac{x}{2\pi}e^{-i\frac{\pi}{2}{\rm sgn\;}t}\right)^w
\frac{(-1)^l}{w(w+1)\cdots(w+l)}\times\\
&\times\frac{\chi_f^{(m-r)}}{\chi_f}(1-s-w)\left(\int_0^\infty\varphi^{(l+1)}(\rho)\rho^{w+l}\sum_{n\leq \rho y}\frac{\lambda_f(n)(-\log n)^r}{n^{s+w}}d\rho\right)
\end{align*}
By the same way to \cite[p.337, Lemma 4 (ii)]{GD1},  ${\rm Res\;}\mathcal{F}$ is written by 
\begin{align}
{\rm Res\;}\mathcal{F}=\sum_{n\leq 2y}\frac{\lambda_f(n)(-\log n)^r}{n^s}\sum_{j=0}^l\varphi^{(j)}\left(\frac{n}{y}\right)\left(-\frac{n}{y}\right)^j\gamma_{j}^{(m-r)}\left(s,\frac{1}{|t|}\right) \label{IC5}
\end{align}
under the condition $x/(2\pi y)=1/|t|$.

Next to estimate $I_{1}'$ and $I_{2}'$, we consider these integral. 
Clearly (\ref{GEE}) gives
\begin{align}
&\frac{\Gamma(s+w+\frac{k-1}{2})}{\Gamma(s+\frac{k-1}{2})}\left(\frac{x}{2\pi}e^{-i\frac{\pi}{2}{\rm sgn\;}t}\right)^w\notag\\
&\ll\begin{cases} |t|^{\frac{1}{2}-\sigma-\frac{k-1}{2}}(1+|t+v|)^{\sigma+u-\frac{1}{2}+\frac{k-1}{2}}(x/2\pi)^u, & w\in L_{\pm 1, \pm 2}, \\ |t|^u(x/2\pi)^u, & w\in\mathcal{F} \end{cases} \label{IE2}
\end{align}
as $|t|\to\infty$. 
Using Cauchy's inequality and (\ref{RAN}), we have
\begin{align*}
\sum_{n\leq\rho y}\frac{\lambda_f(n)(-\log n)^r}{n^{s+w}}
&\ll\sqrt{\sum_{n\leq\rho y}|\lambda_f(n)|^2}\sqrt{\sum_{n\leq\rho y}\frac{(\log n)^{2r}}{n^{2(\sigma+u)}}}\\
&\ll(\rho y)^{1-(\sigma+u)}(\log\rho y)^r, \quad w\in L_{\pm1}\cup C_1,\\
\sum_{n>\rho y}\frac{\lambda_f(n)(-\log n)^r}{n^{s+w}}
&\ll\int_{\rho y}^\infty\left(\frac{(\log\mu)^r}{\mu^{\sigma+u}}\right)'\sum_{n\leq\mu}|\lambda_f(n)|d\mu\\
&\ll(\rho y)^{1-(\sigma+u)}(\log\rho y)^r, \quad w\in L_{\pm2}\cup C_2.
\end{align*}
Hence we obtain
\begin{align}
\int_0^\infty\varphi^{(l+1)}(\rho)\rho^{w+l}\sum_{n\leq \rho y}\frac{\lambda_f(n)(-\log n)^r}{n^{s+w}}d\rho\ll y^{1-(\sigma+u)}(\log y)^r\|\varphi^{(l+1)}\|_1,&\notag\\ 
  w\in L_{\pm1}\cup C_{1},&\label{IE4}\\
\int_0^\infty\varphi^{(l+1)}(\rho)\rho^{w+l}\sum_{n>\rho y}\frac{\lambda_f(n)(-\log n)^r}{n^{s+w}}d\rho\ll y^{1-(\sigma+u)}(\log y)^r\|\varphi^{(l+1)}\|_1,&\notag\\
  w\in L_{\pm2}\cup C_{2}.&\label{IE5}
\end{align}

Therefore Lemma \ref{XAE} gives
\begin{align}
\frac{(-1)^l}{w\cdots(w+l)}\frac{\chi_f^{(m-r)}}{\chi_f}(1-s-w)\ll\begin{cases} |v|^{-(l+1)}(\log|v|)^{m-r}, & w\in L_{\pm1,\pm2},\\ |t|^{-\frac{l+1}{2}}(\log|t|)^{m-r}, & w\in\mathcal{F}.  \end{cases} \label{IE1}
\end{align}
\begin{rem}\label{MEM1}
Note that 
\begin{align*}
&\gamma_j^{(r)}(s,1/|t|)\\
&=
\begin{cases} \displaystyle  O\left(\frac{(\log|t|)^r}{|t|^{j/2}}\right), & j\in\mathbb{Z}_{\geq0},\\
\displaystyle\frac{\chi_f^{(r)}}{\chi_f}(1-s)=\left(-2\log\frac{|t|}{2\pi}\right)^r+O\left(\frac{(\log|t|)^{r-1}}{|t|}\right), & j=0, \\ \displaystyle  \frac{\chi_f^{(r)}}{\chi_f}(1-s)-\frac{\chi_f^{(r)}}{\chi_f}(-s)\frac{it}{s+\frac{k-1}{2}}=O\left(\frac{(\log|t|)^r}{|t|}\right), & j=1, \end{cases}
\end{align*}
by using (\ref{IE1}), the residue theorem and Lemma \ref{XAE}.
\end{rem}

Finally combining (\ref{IE2})--(\ref{IE1}) and using the same way to \cite[p.343--344]{GD1}, we find that $I_1', I_2'$ are estimated as
\begin{align}
I_{1}'\ll& y^{1-\sigma}(\log y)^r\|\varphi^{(l+1)}\|_1\times\notag\\
&\times\int_{L_{\pm1}}|t|^{\frac{1}{2}-(\sigma+u)-\frac{k-1}{2}}(1+|t+v|)^{\sigma+u-\frac{1}{2}+\frac{k-1}{2}}\frac{(\log|v|)^{m-r}}{|v|^{l+1}}dv+\notag\\
&+y^{1-\sigma}(\log y)^r(\log|t|)^{m-r}\|\varphi^{(l+1)}\|_1|t|^{-\frac{l+1}{2}}\int_{C_1}|t|^u\left(\frac{x}{2\pi y}\right)^u|dw|\notag\\
\ll& y^{1-\sigma}(\log y)^r(\log|t|)^{m-r}|t|^{-\frac{l}{2}}\|\varphi^{(l+1)}\|_1, \label{IC7}\\
I_{2}'\ll& y^{1-\sigma}(\log y)^r(\log|t|)^{m-r}|t|^{-\frac{l}{2}}\|\varphi^{(l+1)}\|_1,  \label{IC8}
\end{align}
under the condition $x/(2\pi y)=1/|t|$. From (\ref{IC1})--(\ref{IC5}), (\ref{IC7}) and (\ref{IC8}), the proof of Proposition \ref{PAFE} is completed.
\end{proof}
We use (\ref{STR}) and combine the result Propositions \ref{MAFE} and \ref{PAFE}. Let $y_1$, $y_2$ be the positive numbers satisfying $x/(2\pi y_2)=1/|t|$, $(1/x)/(2\pi y_2)=1/|t|$ respectively. Using Remark \ref{MEM1}, the main term of (\ref{AA1}) is obtained. Then under the condition $(2\pi)^2y_1y_2=|t|^2$, the proof of Theorem \ref{ATH1} is completed.

\section{Proof of Theorem \ref{ATH2}}\label{AMP2}
To get the approximate functional equation for $L_f^{(m)}(s)$ without characteristic functions, we introduce new functions $\xi$, $\psi_{\alpha}$ and $\psi_{0\alpha}$. Let $\xi$ be the function defined by $\xi(\rho)=1$ when $\rho\in[0,1]$ and $\xi(\rho)=0$ when $\rho\in[1,\infty)$. For $\alpha\in\mathbb{R}_{\geq0}$ and $\varphi\in\mathcal{R}$, let $\psi_\alpha$ be the function  defined by 
\begin{align*}
\psi_{\alpha}(\rho)&=\begin{cases} 1,  & \rho\in[0,1-1/(2|t|^{\alpha})],\\ \varphi(1+(\rho-1)|t|^\alpha), & \rho\in[1-1/(2|t|^{\alpha}),1+1/|t|^\alpha],\\ 0, & \rho\in[1+1/|t|^\alpha,\infty), \end{cases}
\end{align*}
and $\psi_{0\alpha}$ is defined by $\psi_{0\alpha}(\rho)=1-\psi_{\alpha}(1/\rho)$. 
\begin{rem}\label{MEM2} From \cite[(12)--(15)]{GD1}, we see that $\psi_{\alpha}, \psi_{0\alpha}\in\mathcal{R}$, $\xi\not\in\mathcal{R}$, 
\begin{align*}
(\psi_\alpha-\xi)(\rho)=0, \quad (\psi_{0\alpha}-\xi)(\rho)=0, \quad \psi^{(j)}_{\alpha}(\rho)=0, \quad \psi^{(j)}_{0\alpha}(\rho)=0.
\end{align*}
for $j\in\mathbb{Z}_{\geq 1}$ and $\rho\in [0,1-1/(2|t|^\alpha)]\cup[1+1/|t|^{\alpha},\infty)$, and
\begin{align*}
\psi_\alpha^{(j)}(\rho)\ll|t|^{\alpha j}, \ \ \psi_{0\alpha}^{(j)}(\rho)\ll|t|^{\alpha j}, \ \ \|\psi_\alpha^{(j)}\|_1\ll|t|^{\alpha(j-1)}, \ \ \|\psi_{0\alpha}^{(j)}\|_1\ll|t|^{\alpha(j-1)}
\end{align*}
for $j\in\mathbb{Z}_{\geq 0}$ and $\rho\in[0,\infty)$.
\end{rem}
Let $M_\varphi(s)$ be the first sum on the right-hand side of (\ref{AA1}).
Setting $y_1=y_2=|t|/(2\pi)$ and replacing $\varphi\mapsto\psi_\alpha$ in Theorem \ref{ATH1}, we can write 
\begin{align}
L^{(m)}_f(s)=M_\xi(s)+O(M_{\psi_\alpha-\xi}(s)+R_{\psi_\alpha}(s)). \label{EE1}
\end{align}
Then we have
\begin{align}
M_\xi(s)=&\sum_{n\leq\frac{|t|}{2\pi}}\frac{\lambda_f(n)(-\log n)^m}{n^{s}}+\notag\\
&+\sum_{r=0}^m(-1)^{m-r}\binom{m}{r}\chi_f^{(m-r)}(s)\sum_{n\leq\frac{|t|}{2\pi}}\frac{\lambda_f(n)(-\log n)^{r}}{n^{1-s}} \label{EE2}
\end{align}
and 
\begin{align}
&\hspace{-1em}
M_{\psi_\alpha-\xi}(s)+R_{\psi_\alpha}(s)\notag\\
\ll&\sum_{\frac{|t|}{2\pi}\frac{1}{1+\frac{1}{|t|^{\alpha}}}\leq n\leq\frac{|t|}{2\pi}(1+\frac{1}{|t|^{\alpha}})}\frac{|\lambda_f(n)|(\log n)^m}{n^\sigma}|S_{\psi_\alpha}^{(0)}(s)|+\notag\\
&+\sum_{r=0}^m\sum_{\frac{|t|}{2\pi}\frac{1}{1+\frac{1}{|t|^{\alpha}}}\leq n\leq\frac{|t|}{2\pi}(1+\frac{1}{|t|^{\alpha}})}\frac{|\lambda_f(n)|(\log n)^r}{n^\sigma}|S_{\psi_{0\alpha}}^{(m-r)}(1-s)|+\notag\\
&+|t|^{1-\sigma+(\alpha-\frac{1}{2})l}(\log|t|)^m. \label{MX1}
\end{align}
where $S_{\psi_\alpha}^{(r)}(s)$ is given by
\begin{align*}
S_{\psi_\alpha}^{(r)}(s)=&(\psi_\alpha-\xi)\left(\frac{2\pi n}{|t|}\right)\frac{\chi_f^{(r)}}{\chi_f}(1-s)+\\
&+\sum_{j=1}^l\psi_\alpha^{(j)}\left(\frac{2\pi n}{|t|}\right)\left(-\frac{2\pi n}{|t|}\right)^j\gamma_{j}^{(r)}\left(s,\frac{1}{|t|}\right),
\end{align*}
and we used Remarks \ref{MEM1}, \ref{MEM2}, (\ref{STR}) and the fact $1-1/(2|t|^\alpha)\geq 1/(1+1/|t|^\alpha)$ for $\alpha\in\mathbb{R}_{\geq0}$.  
Using Remarks \ref{MEM1} and \ref{MEM2}, in the case of $n\in[|t|/(2\pi(1+|t|^{-\alpha})),(1+|t|^{-\alpha})|t|/(2\pi)]$ the sum $S_{\psi_\alpha}^{(r)}(s)$ is estimated as follows under the condition $\alpha\leq 1/2$:
\begin{align}
S_{\psi_\alpha}^{(r)}(s)\ll(\log|t|)^r+\sum_{j=1}^l|t|^{(\alpha-\frac{1}{2})j}(\log|t|)^r\ll (\log|t|)^r\ll|t|^\varepsilon. \label{MX2}
\end{align}
Deligne's estimate $|\lambda_f(n)|\leq d(n)\ll n^\varepsilon$ (see \cite{DEL}) gives
\begin{align}
&\sum_{\frac{|t|}{2\pi}\frac{1}{1+\frac{1}{|t|^{\alpha}}}\leq n\leq\frac{|t|}{2\pi}(1+\frac{1}{|t|^{\alpha}})}\frac{|\lambda_f(n)|(\log n)^r}{n^\sigma}\ll|t|^{1-\sigma-\alpha+\varepsilon}.\label{MX5}
\end{align}
Therefore combining (\ref{MX1})--(\ref{MX5}), we obtain the following estimate:
\begin{align}
M_{\psi_\alpha-\xi}(s)+R_{\psi_\alpha}(s)=O(|t|^{1-\sigma-\alpha+\varepsilon})+O(|t|^{1-\sigma+(\alpha-\frac{1}{2})l+\varepsilon})=O(|t|^{\frac{1}{2}-\sigma+\varepsilon}), \label{EE3}
\end{align}
where we put $\alpha=1/2-\varepsilon$ and take $l\geq1/(2\varepsilon)$. 
Combining (\ref{EE1}), (\ref{EE2}) and (\ref{EE3}), we obtain the assertion of Theorem \ref{ATH2}.

\section{Proof of Theorem \ref{ATH3}}\label{AMP3}
Putting $y_1=y_2=|t|/(2\pi)$ in Theorem \ref{ATH1} and writing $\varphi=\varphi_1+\varphi_2$, $\varphi_0=\varphi_{01}+\varphi_{02}$ where $\varphi_1, \varphi_2, \varphi_{01}, \varphi_{02}$ are defined by (\ref{CFD}), we obtain the following formula: 
\begin{align}
\int_0^T|L_f^{(m)}(s)|^2dt=&\int_1^T\left|\sum_{r=1}^5S_r(s)\right|^2dt+O(1) 
=\sum_{1\leq\mu,\nu\leq 5}I_{\mu,\nu}+O(1), \label{SP1}
\end{align}
where $S_r(s)$ are given by
\begin{align*}
&S_1(s)=\sum_{n=1}^\infty\frac{\lambda_f(n)(-\log n)^m}{n^s}\varphi_1\left(\frac{2\pi n}{t}\right), \quad\\
&S_2(s)=\sum_{n=1}^\infty\frac{\lambda_f(n)(-\log n)^m}{n^{1-s}}\varphi_2\left(\frac{2\pi n}{t}\right),\\
&S_3(s)=\sum_{r=0}^m(-1)^{r}\binom{m}{r}\chi_f^{(m-r)}(s)\sum_{n=1}^\infty\frac{\lambda_f(n)(-\log n)^{r}}{n^{1-s}}\varphi_{01}\left(\frac{2\pi n}{t}\right),\\
&S_4(s)=\sum_{r=0}^m(-1)^{r}\binom{m}{r}\chi_f^{(m-r)}(s)\sum_{n=1}^\infty\frac{\lambda_f(n)(-\log n)^{r}}{n^{1-s}}\varphi_{02}\left(\frac{2\pi n}{t}\right), \quad\\
&S_5(s)=R_\varphi(s),  
\end{align*}
and $I_{\mu,\nu}$ ($\mu,\nu\in\{1,\dots,5\}$) are given by
\begin{align*}
I_{\mu,\nu}=\int_1^TS_\mu(s)\overline{S_\nu(s)}dt.
\end{align*}

First we consider the integral $I_{\mu,\nu}$ in the case of $\mu=\nu$. 
In the case of $(\mu,\nu)=(1,1)$, applying  (\ref{XXE}), (\ref{XYN}) of Lemma \ref{GIE}, we get  
\begin{align}
I_{1,1}
=&\sum_{n_1,n_2=1}^\infty\frac{\overline{\lambda_f(n_1)}\lambda_f(n_2)(\log n_1\log n_2)^m}{(n_1n_2)^\sigma}\times\notag\\
&\times\int_1^T\overline{\varphi_1\left(\frac{2\pi n_1}{t}\right)}\varphi_1\left(\frac{2\pi n_2}{t}\right)\left(\frac{n_1}{n_2}\right)^{it}dt\notag\\
=&T\sum_{n\leq\frac{\delta_1}{2\pi}T}\frac{|\lambda_f(n)|^2(\log n)^{2m}}{n^{2\sigma}}+O\left(\sum_{n\leq\frac{\delta_1}{2\pi}T}\frac{|\lambda_f(n)|^2(\log n)^{2m}}{n^{2\sigma-1}}\right)+\notag\\
&+\frac{1}{i}\sum_{\begin{subarray}{c} n_1,n_2<\frac{\delta_1}{2\pi}T,\\ n_1\ne n_2\end{subarray}}\frac{\overline{\lambda_f(n_1)\varphi_1({2\pi n_1}/{T})n_1^{-iT}}\lambda_f(n_2)\varphi_1({2\pi n_2}/{T})n_2^{-iT}}{(n_1n_2)^\sigma}\times\notag\\
&\times\frac{\log(n_1/n_2)}{(\log n_1\log n_2)^m} +O\left(\sum_{n_1<n_2\leq\frac{\delta_1}{2\pi}T}\frac{|\lambda_f(n_1)\lambda_f(n_2)|(\log{n_1}\log{n_2})^m}{(n_1n_2)^\sigma n_2(\log(n_1/n_2))^2}\right)\notag\\
=&:U_1+O(U_2)+U_3+O(U_4). \label{I11}
\end{align}
Here we shall calculate the right-hand side of (\ref{I11}). Using partial summation and (\ref{RAN}), we obtain the approximate formula for $U_1$  as
\begin{align}
U_1&=\begin{cases}\dfrac{C_f}{2m+1}T(\log T)^{2m+1}+O(T), & \sigma=1/2, \\\displaystyle T\sum_{n=1}^\infty\frac{|\lambda_f(n)|^2(\log n)^{2m}}{n^{2\sigma}}+O(T^{2(1-\sigma)}(\log T)^{2m}), & \sigma\in(1/2,1]. \end{cases}\label{I11A}
\end{align}
The result (\ref{RAN}), the estimates (d), (c) of Lemma \ref{G6A} imply that
\begin{align}
U_j&=\begin{cases} O(T^{2(1-\sigma)}(\log T)^{2m}), & \sigma\in[1/2,1), \\ O((\log T)^{2m+2}), & \sigma=1 \end{cases}\label{I11B}
\end{align}
for $j=2,3,4$ respectively. 
From (\ref{I11})--(\ref{I11B}), the error term and the main term of $I_{1,1}$ correspond to those of the right-hand side of (\ref{AA3}) when $\sigma\in(1/2,1]$. However, the main term of the right-hand side of (\ref{AA3}) is not obtained completely when $\sigma=1/2$. 
In the case of $(\mu,\nu)=(2,2)$, applying  (\ref{XYE}), (\ref{XYN}) of Lemma \ref{GIE} and (a), (c), (d) of Lemma \ref{G6A}, we obtain 
\begin{align}
I_{2,2}
=&\frac{1}{i}\sum_{\begin{subarray}{c} n_1,n_2<\frac{T}{\pi},\\ n_1\ne n_2\end{subarray}}\frac{\overline{\lambda_f(n_1)\varphi_2({2\pi n_1}/{T})n_1^{-iT}}\lambda_f(n_2)\varphi_2({2\pi n_2}/{T})n_2^{-iT}}{(n_1n_2)^\sigma}\times\notag\\
&\times\frac{(\log n_1\log n_2)^m}{\log(n_1/n_2)}+O\left(\sum_{n_1<n_2\leq\frac{T}{\pi}}\frac{|\lambda_f(n_1)\lambda_f(n_2)|(\log{n_1}\log{n_2})^m}{(n_1n_2)^\sigma n_2(\log(n_1/n_2))^2}\right)+\notag\\
&+O\left(\sum_{n\leq\frac{T}{\pi}}\frac{|\lambda_f(n)|^2(\log n)^{2m}}{n^{2\sigma-1}}\right)\notag\\
=&\begin{cases} O(T^{2(1-\sigma)}(\log T)^{2m}), & \sigma\in[1/2,1),\\ O((\log T)^{2m+1}), & \sigma=1.\end{cases}\label{I22} 
\end{align}

Next we consider the case $(\mu,\nu)=(3,3)$. Using (\ref{IEJ}) and the condition $r_1+r_2=r$, we obtain the following formula:
\begin{align*}
(\overline{\chi_f^{(m-r_1)}}\chi_f^{(m-r_2)})(s)
=&(2\pi)^{4\sigma-2}(-2)^{2m-r}\frac{\left(\log\frac{t}{2\pi}\right)^{2m-r}}{t^{4\sigma-2}}+M(t).
\end{align*}
where $M(t)$ is given by $M(t)=O((\log t)^{2m-r}/t^{4\sigma-1})$. Then $I_{3,3}$ is written as 
\begin{align}
I_{3,3}=&\sum_{r=0}^{2m}\sum_{r_1+r_2=r}(-1)^r\binom{m}{r_1}\binom{m}{r_2}\sum_{n_1,n_2=1}^\infty\frac{\overline{\lambda_f(n_1)}\lambda_f(n_2)(\log n_1)^{r_1}(\log n_2)^{r_2}}{(n_1n_2)^{1-\sigma}}\times\notag\\
&\times\int_1^T\overline{\varphi_{01}\left(\frac{2\pi n_1}{t}\right)}\varphi_{01}\left(\frac{2\pi n_2}{t}\right)\left(\frac{n_1}{n_2}\right)^{it}(\overline{\chi_f^{(m-r_1)}}\chi_f^{(m-r_2)})(s)dt\notag\\
=&I_{3,3}^{+}+I_{3,3}^{-}, \label{I33}
\end{align}
where $I_{3,3}^{+}$, $I_{3,3}^{-}$ are defined by
\begin{align*}
I_{3,3}^{+}:=&(2\pi)^{4\sigma-2}\sum_{r=0}^{2m}(-2)^{2m-r}\sum_{r_1+r_2=r}\binom{m}{r_1}\binom{m}{r_2}\times\\
&\times\sum_{n_1, n_2=1}^\infty\frac{\overline{\lambda_f(n_1)}\lambda_f(n_2)(\log n_1)^{r_1}(\log n_2)^{r_2}}{(n_1n_2)^{1-\sigma}}\times\\
&\times\int_1^T\overline{\varphi_{01}\left(\frac{2\pi n_1}{t}\right)}\varphi_{01}\left(\frac{2\pi n_2}{t}\right)\left(\frac{n_1}{n_2}\right)^{it}\frac{\left(\log\frac{t}{2\pi}\right)^{2m-r}}{t^{4\sigma-2}}dt,\\
I_{3,3}^{-}:=&\sum_{r=0}^{2m}\sum_{r_1+r_2=r}\binom{m}{r_1}\binom{m}{r_2}\sum_{\begin{subarray}{c} n_1, n_2=1\end{subarray}}^\infty\frac{\overline{\lambda_f(n_1)}\lambda_f(n_2)}{(n_1n_2)^{1-\sigma}}\times\\
&\times(\log n_1)^{r_1}(\log n_2)^{r_2}\int_1^T\overline{\varphi_{01}\left(\frac{2\pi n_1}{t}\right)}\varphi_{01}\left(\frac{2\pi n_2}{t}\right)\left(\frac{n_1}{n_2}\right)^{it}M(t)dt.
\end{align*}
respectively. Here we shall approximate $I_{3,3}^{+}$ and $I_{3,3}^{-}$. In order to estimate $I_{3,3}^{-}$, we use the fact that  $(n_1n_2)^{1-\sigma} n_2^{4\sigma-2}=(n_1n_2)^\sigma(n_2/n_1)^{2\sigma-1}$ $\gg$ $(n_1n_2)^\sigma$ for $\sigma\in\mathbb{R}_{\geq1/2}$ and $n_1\leq n_2$.  
Then using (d) of Lemma \ref{GIE} and (a), (b) of Lemma \ref{G6A}, we see that 
\begin{align}
I_{3,3}^{-}\ll&\sum_{r=0}^{2m}\sum_{r_1+r_2=r}\sum_{n_1\leq n_2\leq\frac{\delta_1}{2\pi}T}\frac{|\lambda_f(n_1)\lambda_f(n_2)|(\log n_1)^{r_1}(\log n_2)^{r_2}}{(n_1n_2)^{1-\sigma}}\times\notag\\
&\times\begin{cases} |\log(T/n_2)|(\log T)^{2m-r}, & \sigma=1/2, \\ (\log n_2)^{2m-r}/n_2^{4\sigma-2}, & \hspace{-1.25em}\sigma\in(1/2,1]  \end{cases}\notag\\
\ll&\begin{cases} T^{2(1-\sigma)}(\log T)^{2m}, & \sigma\in[1/2,1), \\ (\log T)^{2m+2}, & \sigma=1. \end{cases}\label{I33A}
\end{align}
The formula (\ref{XXE}), (\ref{XYN}) of Lemma \ref{GIE} imply that 
\begin{align}
I_{3,3}^{+}
=&\begin{cases} \displaystyle
 \frac{(2\pi)^{4\sigma-2}}{3-4\sigma}T^{3-4\sigma}\sum_{r=0}^{2m}\left(2\log\tfrac{T}{2\pi}\right)^{2m-r}\times & \\\displaystyle\times\sum_{r_1+r_2=r}{\binom{m}{r_1}\binom{m}{r_2}}\sum_{n\leq\frac{\delta_1}{2\pi}T}\frac{|\lambda_f(n)|^2(\log n)^{r}}{n^{2(1-\sigma)}}, & \sigma\in[1/2,4/3), \\ 0, & \sigma\in[3/4,1],\end{cases}+\notag\\
&+O\left(\sum_{r=0}^{2m}\sum_{n\leq\frac{\delta_1}{2\pi}T}\frac{|\lambda_f(n)|^2(\log n)^{r}}{n^{2(1-\sigma)}}
\times\left\{\begin{array}{r} 
T^{3-4\sigma}(\log T)^{2m-r}, \\ \sigma\in[1/2,3/4),\\ |\log(T/n)|(\log T)^{2m-r}, \\ \sigma=3/4, \\ (\log T)^{2m-r}/n^{4\sigma-3}, \\ \sigma\in(3/4,1]. \end{array}\right.\right)+\notag\\
&+O\left(\sum_{r=0}^{2m}\sum_{\frac{\delta}{2\pi}T<n\leq \frac{\delta_1}{2\pi}T}\frac{|\lambda_f(n)|^2(\log n)^r}{n^{2(1-\sigma)}}\frac{(\log T)^{2m-r}}{n^{4\sigma-3}}\right)+\notag\\
&+\frac{(2\pi)^{4-2\sigma}}{i}\sum_{r=0}^{2m}\left(2\log\tfrac{T}{2\pi}\right)^{2m-r}\sum_{r_1+r_2=r}\binom{m}{r_1}\binom{m}{r_2}\times\notag\\
&\times\sum_{\begin{subarray}{c} n_1, n_2\leq\frac{\delta_1}{2\pi}T, \\ n_1\ne n_2 \end{subarray}}\frac{(\log n_1)^{r_1}(\log n_2)^{r_2}}{(n_1n_2)^\sigma\log(n_1/n_2)}\times\notag\\
&\times\overline{\frac{\lambda_f(n_1)\varphi_{01}({2\pi n_1}/{T})({n_1}/{T})^{2\sigma-1}}{\overline{n_1^{iT}}}}\frac{\lambda_f(n_2)\varphi_{01}({2\pi n_2}/{T})({n_2}/{T})^{2\sigma-1}}{n_2^{iT}}+\notag\\
&+O\left(\sum_{r=0}^{2m}\sum_{r_1+r_2=r}\sum_{n_1<n_2\leq\frac{\delta_1T}{2\pi}}\frac{|\lambda_f(n_1)\lambda_f(n_2)|(\log n_1)^{r_1}(\log n_2)^{r_2}}{(n_1n_2)^{1-\sigma}}\times\right.\notag\\
&\left.\times\frac{(\log n_2)^{2m-r}}{n_2^{4\sigma-1}(\log(n_1/n_2))^2}\right)\notag\\
=&:V_{1}+O(V_{2})+O(V_{3})+V_4+O(V_5), \label{I33B}
\end{align}
A similar discussion to $U_3$ gives that $V_1$ is approximated as
\begin{align}
V_1=\begin{cases} (A_{f,m}-C_f/(2m+1))T(\log T)^{2m+1}+O(T(\log T)^{2m}), & \sigma=1/2, \\ O(T^{2(1-\sigma)}(\log T)^{2m}), & \sigma\in(1/2,1]. \end{cases} \label{I33C}
\end{align}
To estimate $V_4$ and $V_5$, we use the fact that $(n_1n_2)^{1-\sigma}n_2^{4\sigma-1}$ $=$ $(n_1n_2)^\sigma$ $n_2(n_2/n_1)^{2\sigma-1}$ $\gg$ $(n_1n_2)^\sigma n_2$ for $\sigma\in\mathbb{R}_{\geq1/2}$ and $n_1\leq n_2$. Then the estimates (d), (c) of Lemma \ref{G6A} give that  
\begin{align}
V_j=\begin{cases} O(T^{2(1-\sigma)}(\log T)^{2m}), & \sigma\in[1/2,1), \\ O((\log T)^{2m+2}), & \sigma=1 \end{cases} \label{I33D}
\end{align}
for $j=4, 5$ respectively. By the fact that $n^{2(1-\sigma)}\gg  n^{2(1-\sigma)}n^{4\sigma-3}=n^{2\sigma-1}$ for $\sigma\in\mathbb{R}_{\leq 3/4}$, the estimate (b) of Lemma \ref{G6A} when $\sigma=3/4$ and the formula (\ref{RAN}), the sum $V_2$ and $V_3$ are estimated as 
\begin{align}
V_j=\begin{cases} O(T^{2(1-\sigma)}(\log T)^{2m}), & \sigma\in[1/2,1), \\ O((\log T)^{2m+1}), & \sigma=1 \end{cases} \label{I33E}
\end{align}
for $j=2,3$.
Therefore, from (\ref{I33})--(\ref{I33E}) the approximate formula for $I_{3,3}$ is obtained. In the case of $(\mu,\nu)=(4,4)$, by a similar discussion to the case of $(\mu,\nu)=(3,3)$ the integral $I_{4,4}$ is approximated as
\begin{align}
I_{4,4}
=&O\left(\sum_{r=0}^{2m}\sum_{r_1+r_2=r}\sum_{n\leq\frac{T}{\pi}}\frac{|\lambda_f(n)|^2(\log n)^{2m}}{n^{2\sigma-1}}\right)+\notag\\
&+\frac{(2\pi)^{4-2\sigma}}{i}\sum_{r=0}^{2m}\left(2\log\tfrac{T}{2\pi}\right)^{2m-r}\sum_{r_1+r_2=r}\binom{m}{r_1}\binom{m}{r_2}\times\notag\\
&\times\sum_{\begin{subarray}{c} n_1, n_2\leq\frac{T}{\pi}, \\ n_1\ne n_2 \end{subarray}}\frac{(\log n_1)^{r_1}(\log n_2)^{r_2}}{(n_1n_2)^\sigma\log(n_1/n_2)}\times\notag\\
&\times\overline{\frac{\lambda_f(n_1)\varphi_{02}({2\pi n_1}/{T})({n_1}/{T})^{2\sigma-1}}{\overline{n_1^{iT}}}}\frac{\lambda_f(n_2)\varphi_{02}({2\pi n_2}/{T})({n_2}/{T})^{2\sigma-1}}{n_2^{iT}}+\notag\\
&+O\left(\sum_{r=0}^{2m}\sum_{r_1+r_2=r}\sum_{n_1<n_2\leq\frac{T}{\pi}}\frac{|\lambda_f(n_1)\lambda_f(n_2)|(\log n_1)^{r_1}(\log n_2)^{r_2}}{(n_1n_2)^{1-\sigma}}\right.\times\notag\\
&\left.\times\frac{(\log n_2)^{2m-r}}{n_2^{4\sigma-1}(\log(n_1/n_2))^2}\right)+\notag\\
&+O\left(\sum_{r=0}^{2m}\sum_{r_1+r_2=r}\sum_{n_1\leq n_2\leq\frac{T}{\pi}}\frac{|\lambda_f(n_1)\lambda_f(n_2)|^2(\log{n_1})^{r_1}(\log n_2)^{r_2}}{(n_1n_2)^{1-\sigma}}\times\right.\notag\\
&\left.\times\begin{cases} |\log(T/n)|(\log T)^{2m-r}, & \sigma=1/2, \\ (\log n_2)^{2m-r}/n_2^{4\sigma-2}, &  \sigma\in(1/2,1] \end{cases}\right)\notag\\
=& \begin{cases} O(T^{2(1-\sigma)}(\log T)^{2m}), & \sigma\in[1/2,1), \\ O((\log T)^{2m+2}), & \sigma=1, \end{cases}\label{I44X}
\end{align}
where (\ref{XYE})--(\ref{XYM}) of Lemma \ref{GIE}, the formla (\ref{RAN}) and (b)--(d) of Lemma \ref{G6B} were used.
Finally we consider the case $(\mu,\nu)=(5,5)$.   Remarks \ref{MEM1}, \ref{MEM2} and the formula (\ref{STR}) imply that 
\begin{align}
R_\varphi(s)\ll&\sum_{\frac{t}{4\pi}\leq n\leq\frac{t}{\pi}}\frac{|\lambda_f(n)|(\log n)^m}{n^\sigma}\left(\frac{1}{|t|}+\sum_{j=2}^l\frac{1}{|t|^{\frac{j}{2}}}\right)+|\chi_f(s)|\sum_{r=0}^m\sum_{\frac{t}{4\pi}\leq n\leq\frac{t}{\pi}}1\times\notag\\
&\times\frac{|\lambda_f(n)|(\log n)^r}{n^{1-\sigma}}\left(\frac{(\log|t|)^{m-r}}{|t|}+\sum_{j=2}^l\frac{(\log|t|)^{m-r}}{|t|^{\frac{j}{2}}}\right)+\frac{(\log|t|)^m}{|t|^{\sigma-1+\frac{l}{2}}}\notag\\
\ll&\frac{(\log t)^m}{t^\sigma}.  \label{RRE}
\end{align}
Hence we get
\begin{align}
I_{5,5}\ll \int_1^T\frac{(\log t)^{2m}}{t^{2\sigma}}dt \ll\begin{cases} (\log T)^{2m+1}, & \sigma=1/2, \\ 1, & \sigma\in(1/2,1]. \end{cases} \label{I55}
\end{align}

Lastly we consider $I_{\mu,\nu}$ in the case of $\mu\ne\nu$. Since $I_{1,1}$ contains the main term of the mean value formula for $L_f^{(m)}(s)$, and Cauchy's inequality implies that $|I_{\mu,\nu}|\leq I_{\mu,\mu}I_{\nu,\nu}$ for $\mu,\nu\in\{1,\dots,5\}$, it follows that it is enough to consider $I_{\mu,\nu}$ in the case of $(\mu,\nu)=(1,2)$, $(1,3)$, $(1,4)$, $(1,5)$. 
First in the case of $(\mu,\nu)=(1,2)$, using  (\ref{XYE}), (\ref{XYN}) of Lemma \ref{GIE}, (c), (d) of Lemma \ref{G6A} and the estimate (\ref{I11A}), we  obtain 
\begin{align}
I_{1,2}
=&\sum_{n_1, n_2=1}^\infty\frac{\overline{\lambda_f(n_1)}\lambda_f(n_2)(\log n_1)^m(\log n_2)^m}{(n_1n_2)^\sigma}\times\notag\\
&\times\int_1^T\overline{\varphi_1\left(\frac{2\pi n_1}{t}\right)}\varphi_2\left(\frac{2\pi n_2}{t}\right)\left(\frac{n_1}{n_2}\right)^{it}dt\notag\\
=&\frac{1}{i}\sum_{\begin{subarray}{c} n_1,n_2<\frac{T}{\pi},\\ n_1\ne n_2\end{subarray}}\frac{\overline{\lambda_f(n_1)\varphi_{1}({2\pi n_1}/{T})n_1^{-iT}}\lambda_f(n_2)\varphi_2({2\pi n_2}/{T})n_2^{-iT}}{(n_1n_2)^\sigma}\times\notag\\
&\times\frac{(\log n_1\log n_2)^m}{\log(n_1/n_2)}
+O\left(\sum_{n_1<n_2\leq\frac{T}{\pi}}\frac{|\lambda_f(n_1)\lambda_f(n_2)|(\log n_1\log n_2)^m}{(n_1n_2)^\sigma n_2(\log(n_1/n_2))^2}\right)+\notag\\
&+O\left(\sum_{n<\frac{T}{\pi}}\frac{|\lambda_f(n)|^{2}(\log n)^{2m}}{n^{2\sigma-1}}\right)\notag\\
=&\begin{cases} O(T^{2(1-\sigma)}(\log T)^{2m}), & \sigma\in[1/2,1), \\ O((\log T)^{2m+2}), & \sigma=1. \end{cases} \label{I12}
\end{align}
Next we consider the case $(\mu,\nu)=(1,3)$. From (\ref{XYS}) of Lemma \ref{GIE} and (a), (b) of Lemma \ref{G6B}, the integral $I_{1,3}$ is estimated as
\begin{align}
I_{1,3}
=&\sum_{r=0}^m(-1)^m\binom{m}{r}\sum_{n_1, n_2=1}^\infty\frac{\overline{\lambda_f(n_1)}\lambda_f(n_2)(\log n_1)^m(\log n_2)^m}{n_1^\sigma n_2^{1-\sigma}}\times\notag\\
&\times\int_1^T\overline{\varphi_1\left(\frac{2\pi n_1}{t}\right)}\varphi_{01}\left(\frac{2\pi n_2}{t}\right)({n_1}{n_2})^{it}\chi_f^{(m-r)}(s)dt\notag\\
=&O\left(\sum_{r=0}^{2m}\sum_{r_1+r_2=r}\sum_{n_1\leq n_2\leq\frac{\delta_1}{2\pi}T}\frac{|\lambda_f(n_1)\lambda_f(n_2)|(\log n_1)^{r_1}(\log n_2)^{r_2}}{n_1^\sigma n_2^{1-\sigma}}\times\right.\notag\\
&\left.\times\begin{cases} |\log(T/n_2)|(\log T)^{2m-r}, & \sigma=1/2, \\ (\log n_2)^{2m-r}/n_2^{2\sigma-1}, & \sigma\in(1/2,1] \end{cases}\right)\notag\\
=&\begin{cases} O(T^{2(1-\sigma)}(\log T)^{2m}), & \sigma\in[1/2,1), \\ O((\log T)^{2m+2}), & \sigma=1. \end{cases} \label{I13}
\end{align}
In the case of $(\mu,\nu)=(1,4)$, a similar discussion to the case of $(\mu,\nu)=(1,3)$ gives that 
\begin{align}
I_{1,4}=\begin{cases} O(T^{2(1-\sigma)}(\log T)^{2m}), & \sigma\in[1/2,1), \\ O((\log T)^{2m+2}), & \sigma=1. \end{cases}\label{I14}
\end{align}
Finally we consider the case $(\mu,\nu)=(1,5)$. The formula (\ref{RAN}) and Cauchy's inequality imply that $\sum_{n\leq x}|\lambda_f(n)|=O(x)$. 
Then using the estimate (\ref{RRE}) and partial summation we get
\begin{align}
I_{1,5}&\ll\int_1^T\frac{(\log t)^{m}}{t^\sigma}\sum_{n\leq\frac{\delta_1}{2\pi}t}\frac{|\lambda_f(n)|(\log n)^{m}}{n^{\sigma}}dt\notag\\
&\ll\int_1^T\frac{(\log t)^{m}}{t^\sigma}\begin{cases}t^{1-\sigma}(\log t)^m, & \sigma\in[1/2,1), \\ (\log t)^{m+1}, & \sigma=1 \end{cases}dt\notag\\
&\ll\begin{cases} T^{2(1-\sigma)}(\log T)^{2m}, & \sigma\in[1/2,1), \\ (\log T)^{2m+2}, & \sigma=1. \end{cases} \label{I15}
\end{align}
Therefore combining (\ref{SP1})--(\ref{I15}), we complete the proof of Theorem \ref{ATH3}.  \hfill$\square$\\


\end{document}